\documentclass[10pt,twocolumn,twoside]{IEEEtran} 							

\usepackage{amsthm}     
\IEEEoverridecommandlockouts

\usepackage[english]{babel}                 
\usepackage{amsmath, amssymb, amsfonts,hyperref}     
\usepackage{dsfont}
\usepackage{graphicx}                       
\usepackage{color,bbm}                          
\usepackage[normalem]{ulem}                           
\usepackage{tikz}
\usepackage{soul}

\usepackage{subfigure}

\newcommand{\mc}{\mathcal}
\newcommand{\ov}{\overline}
\newcommand{\ds}{\displaystyle}
\newcommand{\ba}{\begin{array}}\newcommand{\ea}{\end{array}}
\newcommand{\de}{\mathrm{d}}
\newcommand{\A}{\mathcal{A}}
\newcommand{\B}{\mathcal{B}}
\newcommand{\C}{\mathcal{C}}
\newcommand{\D}{\mathcal{D}}
\newcommand{\E}{\mathcal{E}}
\newcommand{\F}{\mathcal{F}}
\newcommand{\G}{\mathcal{G}}
\newcommand{\K}{\mathcal{K}}
\newcommand{\I}{\mathcal{I}}

\renewcommand{\O}{\mathcal{O}}

\newcommand{\Q}{\mathcal{Q}}
\newcommand{\R}{\mathcal{R}}
\renewcommand{\S}{\mathcal{S}}
\newcommand{\V}{\mathcal{V}}
\newcommand{\U}{\mathcal{U}}

\newcommand{\W}{\mathcal{W}}
\newcommand{\Z}{\mathcal{Z}}
\newcommand{\Zi}{\Z_{\text{i}}}
\newcommand{\Zo}{\Z_{\text{o}}}
\DeclareMathOperator*{\argmax}{argmax}

\newcommand{\Rb}{\R^{\bullet}}

\newcommand{\RR}{\mathbb{R}}

\newcommand{\se}{\text{if}}

\newcommand{\sgn}[1]{\mathrm{sgn}\left(#1\right)}

\newcommand{\until}[1]{\{1,\dots,#1\}}

\newcommand{\trho}{\tilde{\rho}}
\newcommand{\fin}{f^{\mathrm{in}}}
\newcommand{\fout}{f^{\mathrm{out}}}

\newcommand{\rhomax}{B}
\newcommand{\fmax}{C}
\newcommand{\be}{\begin{equation}}
\newcommand{\ee}{\end{equation}}

\newcommand{\onebf}{\mathbf{1}}
\newcommand{\zerobf}{\mathbf{0}}

\newcommand{\newmaterial}[1]{{#1}}

\newcommand{\comparison}[1]{}

\newcommand{\setdef}[2]{\{#1 \; | \; #2\}}

\newcommand{\defineas}{:=}

\newtheorem{theorem}{Theorem}
\newtheorem{proposition}{Proposition}
\newtheorem{definition}{Definition}
\newtheorem{assumption}{Assumption}
\newtheorem{lemma}{Lemma}
\newtheorem{remark}{Remark}
\newtheorem{example}{Example}

\graphicspath{{images/}}

\title{\LARGE \bf
Throughput optimality and overload behavior \\ of dynamical flow networks under monotone distributed routing}

\author{Giacomo Como, Enrico Lovisari, and  Ketan Savla
\thanks{ G.~Como and E.~Lovisari are with the Department of Automatic Control, Lund University, SE-221 00 Lund, Sweden {\tt\small {giacomo.como, enrico.lovisari}@control.lth.se}. K.~Savla is with the Sonny Astani Department of Civil and Environmental Engineering, University of Southern California, Los Angeles, CA 90089-2531 {\tt\small ksavla@usc.edu}. The first two authors are members of the excellence centers LCCC and ELLIT and were partially supported
by the Swedish Research Council through the Junior Research Grant Information Dynamics in Large Scale Networks.}%
}

\begin{document}

\maketitle
\begin{abstract}
\newmaterial{
The paper investigates the throughput behavior of single-commodity dynamical flow networks governed by monotone distributed routing policies. The networks are modeled as systems of ODEs based on mass conversation laws on directed graphs with limited flow capacities on the links and constant external inflows at certain origin nodes. Under monotonicity assumptions on the routing policies, it is proven that a globally asymptotically stable equilibrium exists so that the network achieves maximal throughput, provided that no cut capacity constraint is violated by the external inflows. On the contrary, should such a constraint be violated, the network overload behavior is characterized. In particular, it is established that there exists a cut with respect to which the flow densities on every link grow linearly over time (resp. reach their respective limits simultaneously) in the case where the buffer capacities are infinite (resp. finite). The results employ an $l_1$-contraction principle for monotone dynamical systems.}
\end{abstract}



\section{Introduction}
\label{section:Introduction}

Rapid technological advancements are facilitating real-time control of infrastructure networks, such as transportation, in order to achieve the efficient utilization of these networks. While static network flows, e.g., see \cite{Ahuja.Magnanti.ea:93}, have traditionally dominated the modeling framework for infrastructure networks, the true potential of the emerging technologies can only be realized by developing control design within a dynamical framework.

In this paper, we study single-commodity dynamical flow networks, modeled as systems of ordinary differential equations derived from mass conservation laws on directed graphs having constant external inflow at each of possibly multiple origins. The state of the system is the density of particles on the links of the network, limited by possibly finite buffer capacities. The flow of particles from a link to downstream links, limited by the maximum flow capacity, is regulated by deterministic rules, or routing policies, which depend on the state of the network. Particles leave the network when they hit at any of the possibly multiple destination nodes. We focus on routing policies that are distributed: the routing at each link only depends on local information consisting of density of itself and the links downstream to it. More specifically, we propose a novel class of \emph{monotone distributed routing policies} that are characterized by general monotonicity assumptions on the sensitivity of their action with respect to local information. 

Our objective is to prove maximum throughput and to characterize the overload behavior in networks operating under monotone distributed routing policies. Our main result is in the form of a dichotomy. If the external inflow at the origin nodes does not violate any cut capacity constraints, then there exists a globally asymptotically stable equilibrium, and the network achieves maximal throughput. When the external inflow at the origin nodes violates some cut capacity constraint, then the network exhibits the following feature: under infinite buffer capacities, there exists a constraint-violating cut, independent of the initial condition, such that the particle densities on the origin side of the cut grow linearly in time with the least possible slope; under finite buffer capacities, there exists a constraint-violating cut, in general dependent on the initial condition, such that all the links constituting the cut hit their buffer capacities simultaneously. The network is thus operated in the most efficient way, from a throughput perspective, even if the routing policies rely only on local information.

These results rely on the ability of the routing policy to implicitly \emph{back-propagate} congestion effects, allowing flow to be routed towards less congested parts of the network in a timely fashion. While algorithms for distributed computation of maximum network flow are well known (e.g., see \cite{Goldberg.Tarjan:88}) the novelty of our contribution consists in proving throughput optimality for flow dynamics naturally arising in physical networks. The proofs are based on an $l_1$-contraction principle for monotone conservation laws (Lemma~\ref{lemma:l1Contraction}), and on a complete characterization of all possible combinations of limiting (as densities approach the buffer capacities) states of all the links around every node (Lemma~\ref{lemma:propertiesWBCF}). \newmaterial{The former, in particular, is analogous to properties of some hyperbolic partial differential equations: e.g., cf.~Kru\v{z}kov's Theorem \cite[Proposition 2.3.6]{Serre:99} for entropy solutions of scalar conservation laws.}

\newmaterial{
The distributed routing architecture of this paper and the ensuing result on throughput optimality is reminiscent of the back-pressure routing algorithm for multi hop networks~\cite{TassiulasTAC93}. In \cite{TassiulasTAC93}, dynamics is imposed on the nodes of the network, instead of the links as here. However, one can transform our setup to fit within the one of \cite{TassiulasTAC93} by employing a \emph{dual graph} where the roles of nodes and links are exchanged in a suitable manner. The back-pressure routing setup allows for arbitrary constraints on simultaneous activation of links in the network. For specific constraints under which at most one, among all outgoing links, at every node can be activated, then the back-pressure routing, with the max operation replaced with softmax, can be argued to satisfy the properties of monotone distributed routing of this paper. Such an argument also extends to generalizations of back-pressure policies, such as the the MaxWeight-$f$ policies, for strictly increasing $f$. 

The dynamical formulation of this paper is also reminiscent of dynamic traffic flow over networks, e.g., see \cite{Daganzo:95,
Garavello.Piccoli:06}. In particular, our framework can be used to analyze dynamical traffic models that are related to the well-known \emph{cell transmission model} (CTM)~\cite{Daganzo:94,Daganzo:95}. The CTM can be explained for a line network as follows: a line is partitioned in \emph{cells} $e = 1, \dots, N$, in each of which the traffic state is described by traffic density. The system is driven by mass-conservation and the flow from a cell to the following is given by the minimum of two quantities: demand of cell $e$, describing the amount of vehicles that desire to enter into cell $e+1$, and supply of cell $e+1$, describing the maximum amount of vehicles that are allowed into it. 
Such a dynamical setup can be shown to satisfy the monotonicity properties of this paper, and hence one can derive tight conditions on the existence and stability of equilibria in such settings~\cite{LovisariCDC14}. As such, this result is a continuous time counterpart of \cite{GomesTRC08}. The CTM setup is extended to the general network case by specifying fixed \emph{turning ratios} and by imposing FIFO (first-in-first-out) constraints at diverging junctions~\cite{Daganzo:95}. In this case, the resulting setup does not necessarily satisfy the monotonicity properties of this paper. However, by relaxing the FIFO constraints, one recovers the monotonicity properties, and the results of this paper can then be utilized for analysis of such a model~\cite{LovisariCDC14}.
}

It is imperative to highlight the difference between this paper and our previous work~\cite{ComoPartITAC13, ComoPartIITAC13}, where we studied dynamical flow networks in which the action of the routing policy at a node is restricted to splitting the (given) inflow from incoming links at that node among the links outgoing from that node, as a function of the density on outgoing links. Specifically, such a routing architecture did not allow backward propagation of congestion effect. We proposed and studied a class of \emph{locally responsive policies} under such an architecture, for the infinite buffer capacity case and for directed acyclic network topologies. In this paper, we extend and modify such a framework, not only by allowing finite buffer capacities and cyclic network topologies, but more importantly by allowing the routing policies to completely control the flow transfer between links. Under this framework, we are able to provide explicit conditions for global asymptotic stability of equilibria and, unlike \cite{ComoPartITAC13, ComoPartIITAC13}, we give a detailed characterization of the overload behavior of the network.

The paper is organized as follows: Sec.~\ref{sec:motivatingexample} provides a motivating example for the study of monotone distributed policies. In Sec.~\ref{section:dynamicalFlowNetworks}, we propose a general model for dynamical flow in networks. In Sec.~\ref{section:MainResults} we state our main results, which are proven in Sec.~\ref{section:proof}. Finally, Sec.~\ref{section:conclusions} states conclusions and possible directions for future research.

We conclude by introducing some notational conventions. For finite sets $\mc A$ and $\mc B$, $\RR^{\mc A}$ ($\RR_+^{\mc A}$) is the space of real-valued (nonnegative-real-valued) vectors whose entries are indexed by elements of $\mc A$ and $\RR^{\mc A\times\mc B}$ the space of matrices whose real entries are indexed by pairs in $\mc A\times\mc B$. $M' \in\RR^{\mc B\times\mc A}$ is the transpose of $M \in \RR^{\mc A \times\mc B}$. Inequalities such as $x\le y$ or $x< y$ for vectors $x,y\in\RR^{\mc A}$ are meant to hold component-wise. 
We identify a network with a weighted directed multi-graph $\G = (\V, \E, C)$, where $\V$ and $\E$ stand for the finite sets of nodes and links, respectively, and $C\in(0,+\infty]^{\E}$ are link capacities. For link $e\in\mc E$, $\sigma_e$ and $\tau_e$ denote its tail and head nodes, respectively, so $e=(\sigma_e,\tau_e)$. While we assume there are no self-loops, i.e., $\tau_e\ne\sigma_e$ for $e \in \E$, we allow for parallel links.

\newmaterial{
\section{A motivating example}
\label{sec:motivatingexample}

\begin{figure}[t]
	\centering
		\includegraphics[width=9cm]{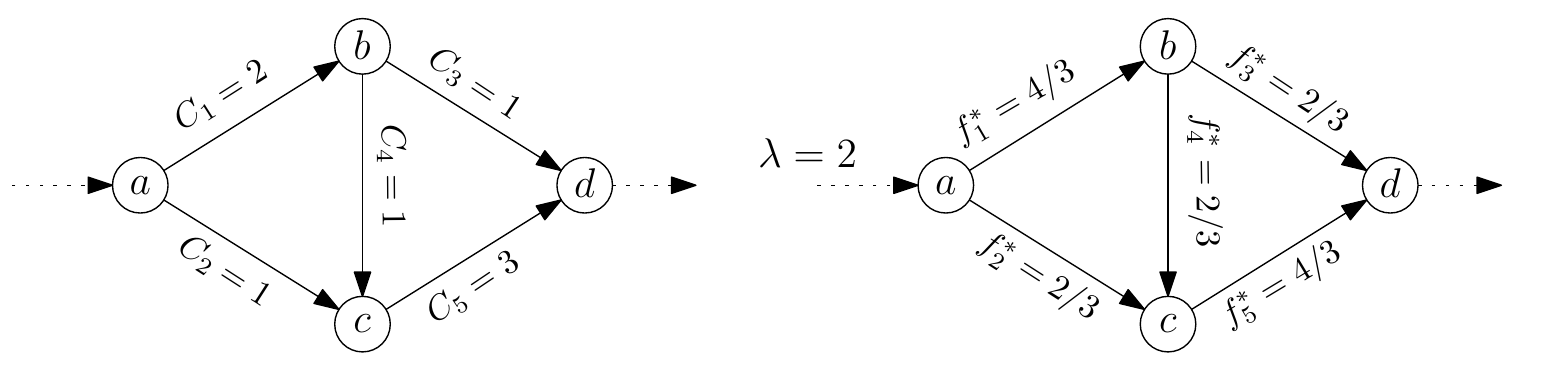}
	\caption{On the left, the network $\mc G$ analyzed in Sec.~\ref{sec:motivatingexample}, with node set $\mc V=\{a,b,c,d\}$, link set $\E=\{1,2,3,4,5\}$, and capacities $C_1 =2$, $C_2=C_3=C_4=1$, $C_5 = 3$. On the right, an equilibrium flow $f^*$.}
	\label{fig:throughputOptimalMotivatingExample}
\end{figure}

Consider the network $\mc G=(\mc V,\mc E,C)$ in Figure \ref{fig:throughputOptimalMotivatingExample}, with inflow $\lambda=2$ from node $a$ and equilibrium flow $f^*$. Our goal is to study throughput and resilience of dynamical flows on $\mc G$. As it turns out, these properties do not depend merely on $\mc G$ and $f^*$, but also on the specific flow dynamics. We focus on first-order dynamics of the form
	\be
		\label{dynsysexample}
			\dot\rho=A(F'(\rho)-F(\rho))\onebf\,,
	\ee
where $\rho=\rho(t)\in\RR_+^{5}$ is the vector of densities on the different links; $F(\rho)\in\RR^{6\times 6}$ is the matrix of link-to-link flows with $F_{ij}(\rho)$ denoting the flow from link $i$ to link $j$ and with the last row and column of $F(\rho)$ corresponding to inflows from and, respectively, outflows to the external world; $A = [I_{5 \times 5} \, \zerobf_{5 \times 1}]$ is the projection matrix on the first $5$ components; and $\onebf\in\RR^{5}$ is the all-one vector. Assume that all the links have infinite buffer capacities, i.e., the range of $\rho_e(t)$ is $[0,+\infty]$ for all $e \in \mc E$. 
To reflect the structure of $\mc G$ and invariance of the nonnegative orthant $\RR_+^5$ for solutions of (\ref{dynsysexample}), assume that  
	$$
			F(\rho)=M(\rho)R(\rho)\,,
	$$
where $M(\rho)=\text{diag}(C_1\varphi(\rho_1),\ldots,C_5\varphi(\rho_5),\lambda)$, with $\varphi(\rho)$ Lipschitz continuous and strictly increasing from $\varphi(0)=0$ to $\lim_{\rho\to\infty}\varphi(\rho)=1$, and $R(\rho)\in\RR^{6\times6}$ is a row-substochastic routing matrix with $R_{ij}(\rho)\equiv0$ whenever $\tau_i\ne\sigma_j$ (with the convention $\sigma_6=d$, $\tau_6=a$). The term $C_i\varphi(\rho_i)$ represents the density-dependent maximal outflow from a link $i$, while $R_{ij}(\rho)$ stands for the fraction of such maximal outflow routed to the downstream link $j$. E.g., dynamics on link $1$ reads $\dot{\rho}_1 = \lambda R_{61}(\rho) - \fmax_1 \varphi(\rho_1)(R_{13}(\rho)+R_{14}(\rho))$. Let 
	$$
		\mu(C,R):=\liminf_{t\to+\infty}\frac1t\int_0^t \left(F_{36}(\rho(s))+F_{56}(\rho(s)) \right) \de s
	$$
denote the throughput of flow dynamics \eqref{dynsysexample}, i.e., the long-term average inflow at the destination node $d$. A perturbation of \eqref{dynsysexample} is a dynamical system with the same network topology and routing matrix, but a potentially different vector of link capacities, $\tilde{\fmax}$. In the following, we shall be interested in measuring how much can such perturbations reduce the throughput of the system. Assume that the unperturbed dynamics \eqref{dynsysexample} admits an equilibrium $\rho^*\in\RR^5$ with
	$$
		AM(\rho^*)R(\rho^*)\onebf=AR'(\rho^*)M(\rho^*)\onebf=f^*\,,
	$$ 
so that in particular $\mu(C,R)=\lambda=2$, and define the \emph{resilience function} $\nu(\,\cdot\,)$ by letting, 
for $\delta\ge0$,
	$$
	\nu(\delta):=\inf_{\zerobf \leq \tilde C \leq C}\Big\{\sum_{1\le i\le 5}(C_i-\tilde C_i):\,\mu(\tilde C,R)< \lambda-\delta\Big\}\,.
	$$

Resilience function has a natural interpretation as the effort required by an adversary, who is choosing $\tilde{C} \leq C$, to cause a throughput loss $\delta$, given that the routing policy is $R(\rho)$.
Note that,  
if $C_{\mc G}=3$  stands for the min-cut capacity of $\mc G$, then necessarily $\nu(\delta)\le C_{\mc G}-\lambda+\delta=1+\delta$ for $0\le\delta\le\lambda$. Indeed, reducing the capacities of the links of a minimal cut in such a way that the perturbed min-cut capacity $\tilde C_{\mc G}$ does not exceed $\lambda$, then the throughput drops from $\lambda$ to at most $\tilde C_{\mc G}$. 


We now characterize the resilience function of three different routing matrices. Let us start with a fixed routing matrix
	$$
		R^{(1)}(\rho)
			\equiv
				\left[
					\ba{cccccc}
						0		&	0		&	1/2	&	1/2	&	0	&	0	\\
						0		&	0		&	0		&	0		&	1	&	0	\\
						0		&	0		&	0		&	0		&	1	&	0	\\
						0		&	0		&	0		&	0		&	0	&	1	\\
						0		&	0		&	0		&	0		&	0	&	1	\\
						2/3	  &	 1/3	&	0		&	0		&	0	&	0
					\ea
				\right]\,.
	$$
Let the network be at equilibrium $f^*$ before the perturbation. Under fixed routing $R^{(1)}$, a capacity reduction on link $3$ does not change its inflow $f_3^*=2/3$. If $f_3^*\ge\tilde C_3$ the density $\rho_3(t)$ cannot but grow unbounded and a throughput loss of $f_3^*-\tilde C_3$ occurs. Thus, the resilience function satisfies $\nu^{(1)}(\delta)\le C_3-f_3^*+\delta=1/3+\delta$ for $0\le\delta<2/3$. Additionally reducing capacity on link $4$ shows that $\nu^{(1)}(\delta)\le C_3+C_4-f_3^*-f_4^*+\delta=2/3+\delta$ for $2/3\le\delta<4/3$, and similarly up to $\delta = 2$. In fact, these bounds can be shown to be tight and the resilience function to be the one plotted in grey in Figure \ref{fig:throughputOptimalNuExample}.

\begin{figure}[t]
	\centering
		\includegraphics[width=5.5cm,height=4cm]{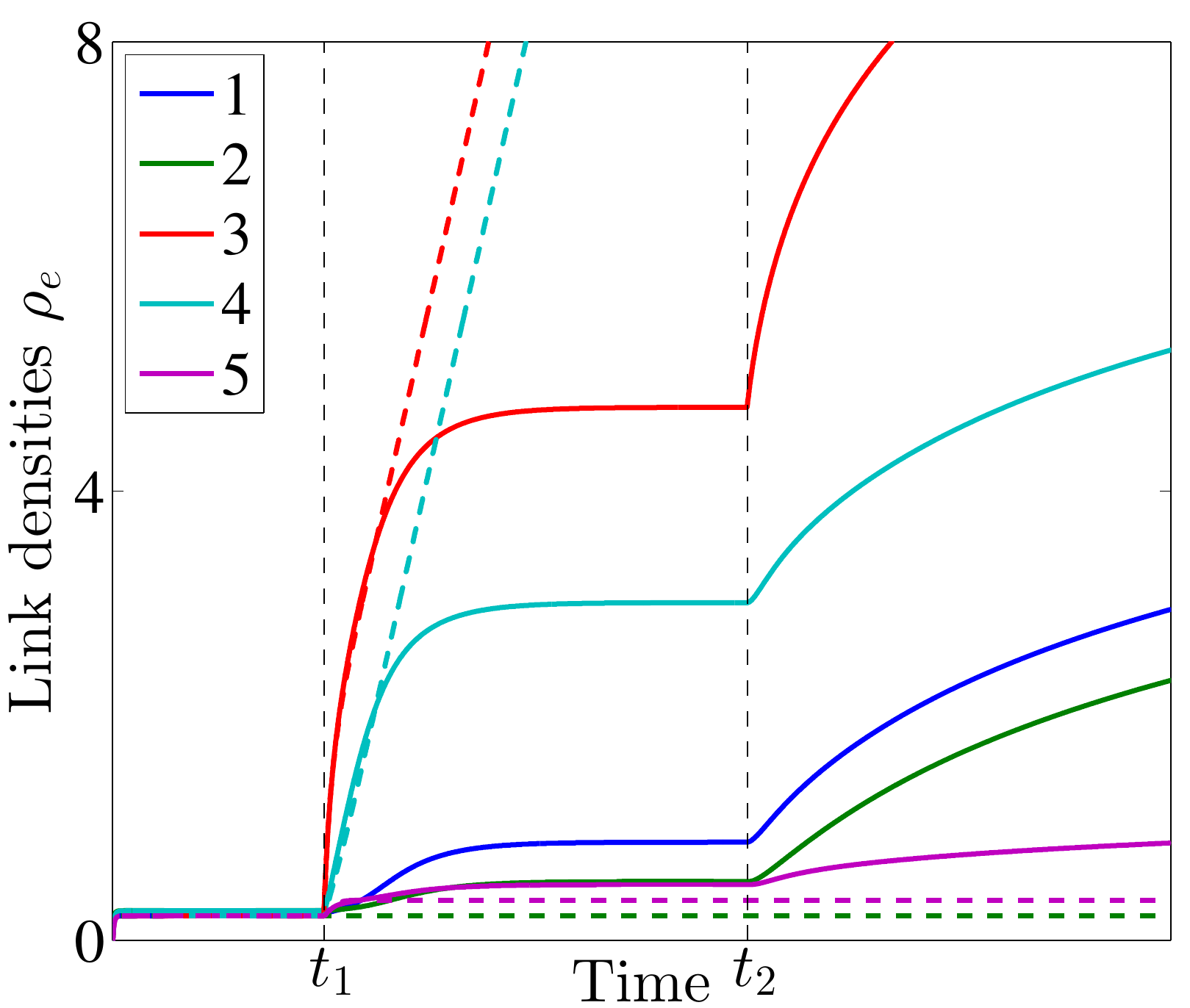}
	\caption{Trajectories of the solutions of \eqref{dynsysexample} under routing matrix $R^{(2)}$ (dashed lines) and $R^{(3)}$ (solid lines), with initial condition $\rho(0) = \zerobf$. At time $t_1$, the capacity on link $3$ drops to $\tilde{C}_3 = 1/6$, and at time $t_2$ it drops to $\tilde{\fmax}_3 = 0$. Under $R^{(2)}$ and after $t_1$, $\tilde{\fmax}_3+\fmax_4 = 7/6$ is smaller than the total outflow from link $1$, so the densities on links $3$ and $4$ grow unbounded. The min-cut capacity $\fmax_2+\tilde{\fmax_3}+\fmax_4 =13/6$ remains instead strictly higher than $\lambda = 2$, so the monotone distributed policy steers the network to a new equilibrium. After $t_2$, the min-cut capacity drops to $2$, the constraint is thus violated, and the densities of links $1$, $2$, $3$ and $4$ grow unbounded.}
	\label{fig:throughputOptimalMotivatingExampleSimulation}
\end{figure}

Now, let $R^{(2)}(\rho)$ be a locally responsive routing matrix~\cite{ComoPartITAC13, ComoPartIITAC13} with all entries coinciding with those of $R^{(1)}(\rho)$ except
	\be
		\ba{rcl}
			\label{R2}
				R^{(2)}_{13}(\rho)	&	=	1-R^{(2)}_{14}(\rho)	=	&	\ds\frac{e^{-\rho_3}}{e^{-\rho_4}+e^{-\rho_3}}\,,\\[10pt]
				R^{(2)}_{61}(\rho)	&	=	1-R^{(2)}_{62}(\rho)	=	&	\ds\frac{2e^{-\rho_1}}{2e^{-\rho_1}+e^{-\rho_2}}\,.
			\ea
	\ee 
In this case, reducing the capacity of link $3$ only does not cause any throughput loss if $\tilde C_3>1/3$. In fact, even if link $3$ cannot handle its initial inflow $f_3^*$, the system is able to adapt by rerouting the flow out of node $b$ and exploit the unused capacity on link $4$, so that a new equilibrium is reached provided that $f^*_1=f^*_3+f^*_4<\tilde C_3+\tilde C_4$. 

However, this is no longer the case if $\tilde C_3\le1/3$, as then $f^*_1\ge\tilde C_3+\tilde C_4$ and both $\rho_3(t)$ and $\rho_4(t)$ necessarily grow unbounded in $t$, with a throughput loss of $f^*_1-\tilde C_3-\tilde C_4$. This shows that the resilience function satisfies $\nu^{(2)}(\delta)\le2/3+\delta$, for $0\le\delta<4/3$. In fact, results in \cite{ComoPartIITAC13} on \emph{diffusivity} of locally responsive routing, i.e., a subadditive property for aggregate outflow increases in subnetworks as a function of capacity reductions, can be used to show that this bound is tight, and that $\nu^{(2)}(\delta)$ has the graph plotted in Figure \ref{fig:throughputOptimalNuExample}. 

\begin{figure}[t]
	\centering
		\includegraphics[height=4cm]{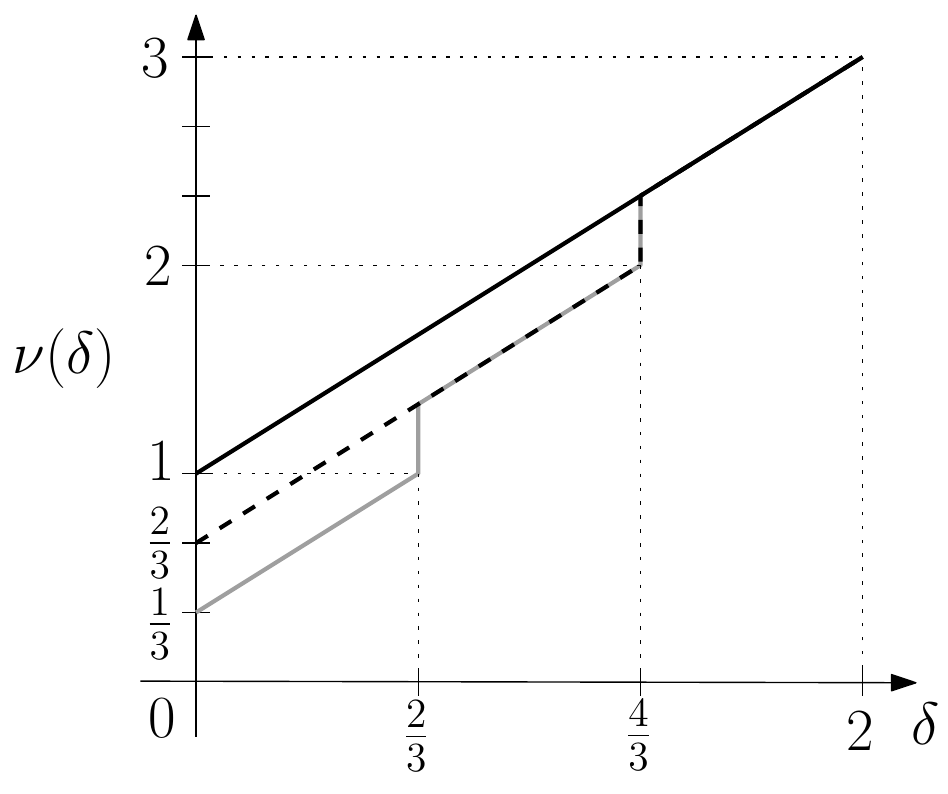}\\
	\caption{Resilience as a function of the throughput loss $\delta$ given the three routing matrices presented in Sec.~\ref{sec:motivatingexample}: the solid grey, dashed  black, and solid black lines	illustrate $\nu^{(1)}(\delta)$,  $\nu^{(2)}(\delta)$, and  $\nu^{(3)}(\delta)$, respectively. }
	\label{fig:throughputOptimalNuExample}
\end{figure}



Finally, consider a routing matrix $R^{(3)}(\rho)$ coinciding with $R^{(2)}(\rho)$ in all  but its $(1,3)$-th and $(1,4)$-th entries, given by 
	$$
		R_{13}^{(3)}(\rho) 
			= R_{13}^{(2)}(\rho)h(\rho)\,,\qquad 
		R_{14}^{(3)}(\rho)
			=	R_{14}^{(2)}(\rho)h(\rho)\,,
	$$ 
where 
	\be
		\label{hdef}
			h(\rho)=\frac{e^{-\rho_4}+e^{-\rho_3}}{e^{-\rho_1}+e^{-\rho_4}+e^{-\rho_3}}\in[0,1]
	\ee 
can be interpreted as a flow control term. Figure~\ref{fig:throughputOptimalMotivatingExampleSimulation} shows the trajectories of the link densities when a perturbation is applied at time $t_1$ such that $\tilde{\fmax}_3=1/6$. Observe that, although $f^*_1=4/3>7/6=\tilde C_3+\tilde C_4$, the link densities remain bounded in time and approach a new equilibrium. In fact, the mechanism allowing the network to absorb the perturbation can be understood rather intuitively: link $3$ is not capable to sustain its initial inflow $f^*_3$, nor links $3$ and $4$ are collectively able to sustain $f^*_1$, thus both $\rho_3(t)$ and $\rho_4(t)$ increase, thereby decreasing $h(\rho)$. This in turn forces the outflow from link $1$ to decrease and hence $\rho_1(t)$ to grow, namely, density increase is \emph{back-propagated} towards the origin, a mechanism that was completely absent in \cite{ComoPartIITAC13}. Then, the dynamic routing at node $a$ redirects more flow towards link $2$, such an increase being still small enough that densities on none of the links grow unbounded. In this way, the network is able to absorb the perturbation and reach a new equilibrium. Indeed, the main results of the present paper imply that in this case the resilience function $\nu^{(3)}(\delta)=1+\delta$ is the maximum possible. Specifically, Theorem \ref{theorem:mainResult} shows that, as long as the inflows do not violate any cut capacity constraints in the network, flow dynamics with the same properties as \eqref{dynsysexample} with routing $R^{(3)}(\rho)$ always admit a globally asymptotically stable equilibrium, while Proposition \ref{proposition:InfiniteBuffer} implies that if the min-cut capacity is smaller than the inflow in the network, either from the beginning or as the result of a perturbation, then the throughput is equal to the min-cut capacity itself, and is thus the maximum possible. 

Finally, it is possible to consider an analogous setting with finite buffer capacities $B_e$, where a link $e$ is irreversibly removed from the network the first time that $\rho_e(t)=B_e$, i.e., when link $e$ fails. While referring to Sec.~\ref{section:dynamicalFlowNetworks} and \cite{ComoCDC12} for a precise formulation of this setting with routing analogous to $R^{(3)}(\rho)$  and $R^{(2)}(\rho)$, respectively, we anticipate here that results paralleling the above-discussed infinite buffer capacity case can be established for the resilience function. Figure \ref{fig:perturbationExamplePolicy} reports the sequence of link failures for the three routing matrices when a perturbation affecting only link $3$ is applied such that $C_3-\tilde C_3=\nu^{(r)}(0)$, for $r = 1,2,3$. For the fixed routing matrix $R^{(1)}$, a perturbation in link $3$ such that $\tilde C_3=2/3$ makes link $3$ to fail first, thus forcing node $b$ to route all its outflow to link $4$ and making it fail, which in turn causes the failure of the upstream link $1$, thus forcing node $a$ to route all its outflow to link $2$ and making it fail. For the routing matrix $R^{(2)}(\rho)$, a perturbation in link $3$ such that $\tilde C_3=1/3$ first forces links $3$ and $4$ to fail simultaneously, then links $1$ and $2$ fail simultaneously since their inflow $\lambda=2$ is not smaller than $C_2$.  In contrast, for the routing matrix $R^{(3)}(\rho)$, a perturbation in link $3$ such that $\tilde C_3=0$ makes links $1$, $2$, $3$ and $4$ fail simultaneously. This is because the routing matrix $R^{(3)}(\rho)$ exploits the available capacity by redistributing the flow in the best way to avoid link failures as long as possible. Observe that the  failed links are those on the origin side of the bottleneck cut consisting of links $2$, $3$, and $4$, whose capacity upon perturbation equals the inflow $\lambda=2$. As we shall see in Proposition~\ref{corollary:FiniteBuffer}, this is a special case of a general result holding true for flow dynamics with the same properties as those of (\ref{dynsysexample}) with routing matrix $R^{(3)}(\rho)$.

\begin{figure}[t]
    \centering
    \subfigure[]
    {
        \includegraphics[height=1.65cm]{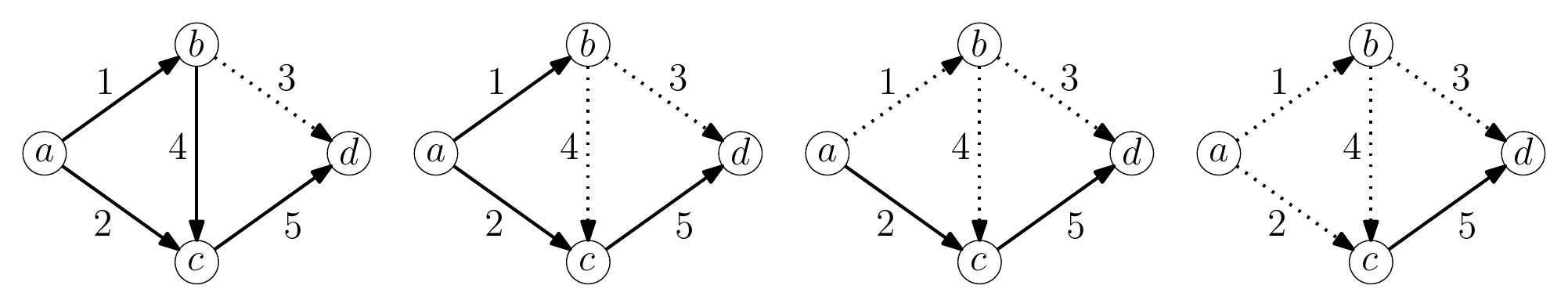}
        \label{fig:fixed}
    }
    \\
    \subfigure[]
    {
        \includegraphics[height=1.65cm]{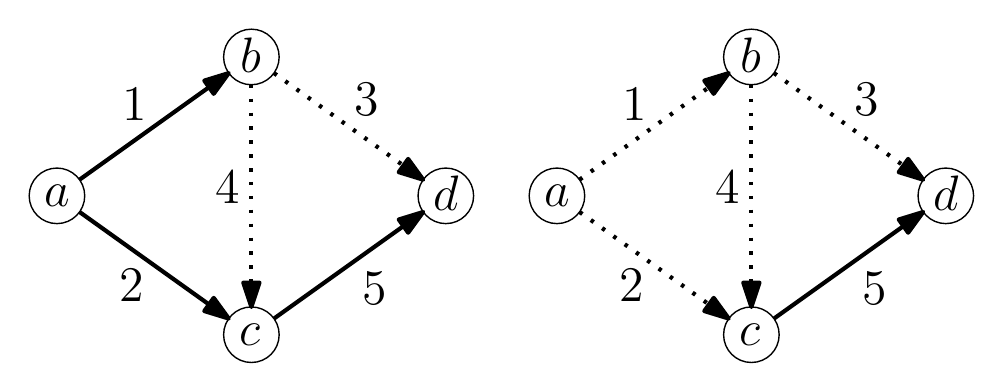}
        \label{fig:dynamically}
    }
    \subfigure[]
    {
        \includegraphics[height=1.65cm]{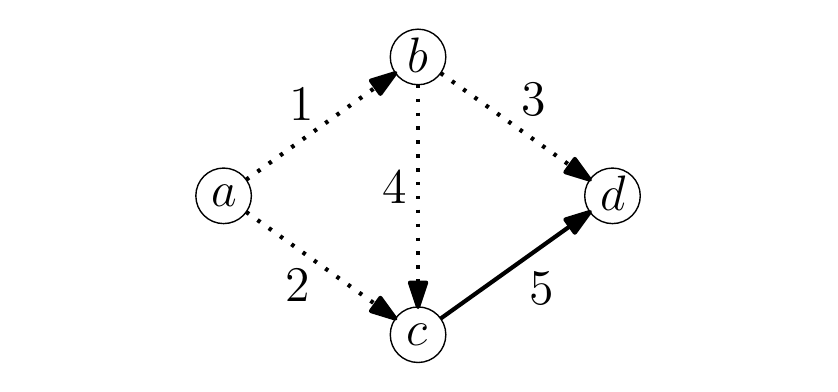}
        \label{fig:monotone}
    }
    \caption{Sequences of link failures under finite buffer capacities using routing matrices analogous to those in Sec.~\ref{sec:motivatingexample}. Link $3$ is perturbed only, with $C_3-\tilde C_3$ equal to $\nu^{(1)}(0)=1/3$, $\nu^{(2)}(0)=2/3$, and $\nu^{(3)}(0)=1$, respectively.
    }
    \label{fig:perturbationExamplePolicy}
\end{figure}
}

\section{Dynamical flow networks with monotone distributed routing}
\label{section:dynamicalFlowNetworks}

\begin{figure}[t]
	\centering
		\includegraphics[width=7cm]{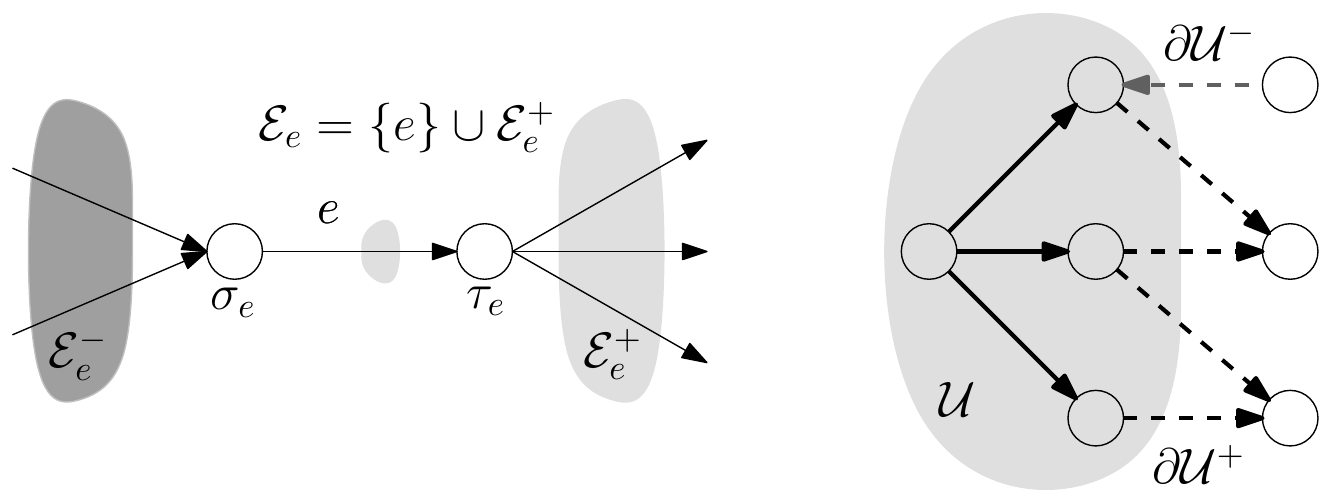}
	\caption{On the left, the dark grey and light grey areas encompass links in $\E_e^-$ and in $\E = \{e\}\cup\E_e^+$, respectively. On the right, the grey area encompasses the nodes of a cut $\U$. Links in $\partial_{\U}^-$ and $\partial_{\U}^+$ are shown in dashed grey and black arrows, respectively; links in $\mc E_\U^+ \setminus \partial_{\U}^+$ are shown in solid arrows.}
	\label{fig:throughputOptimalNotation}
\end{figure}

\newmaterial{
For a network  $\G = (\V, \E, \fmax)$, we introduce the following notation, illustrated in Figure~\ref{fig:throughputOptimalNotation}. Let $\E_v^+ \defineas \{e \in \mc E:\sigma_e = v\}$ and $\E_v^- \defineas \{e \in \mc E:\tau_e = v\}$ be the sets of incoming and, respectively, outgoing links of a node $v$. For a link $e$, let $\E_e^+ \defineas \E_{\tau_e}^+$ and $\E_e^- \defineas \E_{\sigma_e}^-$  be, respectively, the sets of its downstream and upstream links, and let 
	\be
		\label{Eedef}\E_e:=\E^+_e\cup\{e\}\,.
	\ee
For a subset $\mc U\subseteq\mc V$, define 
	$$
		\E_{\U}^+ \defineas \cup_{u\in\U}\E_u^+\,,
		\qquad 
		\E_{\mc U}^- \defineas \cup_{u\in \mc U}\E_u^-\,,
	$$ 
	$$
		\partial_{\U}^+ \defineas \{e: \sigma_e\in \U, \tau_e\notin \U\}\,,\ 
		\partial_{\U}^-:= \{e:\,\sigma_e\in\mc V\setminus\U,\,\tau_e\in \U\}\,.
	$$  

Let $\D \defineas \{v \in \mc V:\,\E_v^+=\emptyset\}$ be the set of \emph{destination nodes}. Consider a \emph{vector of inflows} $\lambda\in\RR_+^{\mc V\setminus\D}$ whose $v$-th entry $\lambda_v$ stands for the external inflow in node $v$, and let $\O \defineas \{v \in \mc V :\,\lambda_v>0\}$ be the set of \emph{origin} nodes. 
Let a cut be a non-empty subset of non-destination nodes $\U\subseteq\V\setminus\D$ and denote its capacity by $C_{\U} \defineas \sum_{e\in\partial_{\U}^+}\fmax_e$ and its aggregate external inflow by $\lambda_{\U}\defineas \sum_{v\in\U}\lambda_v$. 

It proves convenient to introduce the \emph{augmented} network $\G^a = (\V^a,\E^a,C^a)$ (see Figure~\ref{fig:throughputOptimalNetwork}) with node and link sets $\V^a = \V\cup\{w\}$, $\E^a = \E\cup\E_\O^-\cup\E_\D^+$, respectively, where $$\E_\O^- \defineas \{e_v \!:=\! (w, v): v\in\O\},\quad \E^+_\D \defineas \{e^d \!:=\! (d, w): d \in \D\},$$ and $C_{e_v} = C_{e^d} = +\infty$ for all $v \in \O$ and $d \in \D$. The extra node $w$ may be thought of as representing an external world, playing the double role of source of the flow entering in the network at the origins, and sink of the flow exiting from the destinations, respectively. From now on, we adopt the notation $\E_{e_v} = \E_{e_v}^+ := \E_v^+$, for all $v\in\O$,  and let $\E^-_e$ and $\E_v^-$ include links in $\E_\O^-$, and $\E^+_e$ and $\E^+_v$ include links in $\E_\D^+$, thus using these symbols consistently with the augmented graph $\mc G^a$.
Throughout, we shall make the following assumption.
\begin{assumption}
\label{assumption:connectivity} 
The set of destinations $\D$ is nonempty, and the augmented network $\G^a$ is strongly connected.
\end{assumption}
Assumption \ref{assumption:connectivity} is equivalent to the properties that, in $\G$, for every $v\in\V\setminus\D$, there exists a directed path from $v$ to some destination node $d\in\D$ and,  for every $u\in\V\setminus\O$, there exists a directed path from some origin node $o\in\O$ to $u$. Note that Assumption \ref{assumption:connectivity} implies that there is no subset $\A \subseteq \V$ that is unreachable in $\G^a$, i.e., such that $\partial_\A^- = \emptyset$ and $\lambda_\A = 0$.

\begin{remark}
Cut capacities determine potential bottlenecks for network flows. In particular, the max-flow min-cut theorem \cite{FordFulkersonCJM56,EliasIRETIT56} states that
    \begin{equation}
        \label{equation:MaxFlowMinCut}
        \max_{\lambda}\max_{\U}\left\{\lambda_{\U}-C_{\mc U}\right\}=0\,,
    \end{equation}
where the internal maximization runs over all cuts $\mc U$ and the external maximization runs over all external vectors of inflows $\lambda\in\RR_+^{\mc V\setminus\D}$ for which there exists some flow vector $f\in\RR_+^{\E\cup\E^-_{\O}}$ such that $f_{e_o}=\lambda_o$ for $o\in\mc O$, $\sum_{e\in\E^+_v}f_e-\sum_{e\in\E^-_v}f_e=0$ for $v\in\mc V\setminus\mc D$ and $f_e\le C_e$ for $e\in\mc E$.
In the special case of a single  origin $\O = \{o\}$, equation \eqref{equation:MaxFlowMinCut} reduces to $\max \lambda_o = \min_{\U}C_\U$, i.e., the maximum admissible inflow equals the min-cut capacity. Observe that this result is a purely static one as it only concerns potential equilibrium flows. 
\end{remark}

\begin{figure}[t]
	\centering
		\includegraphics[height = 4cm]{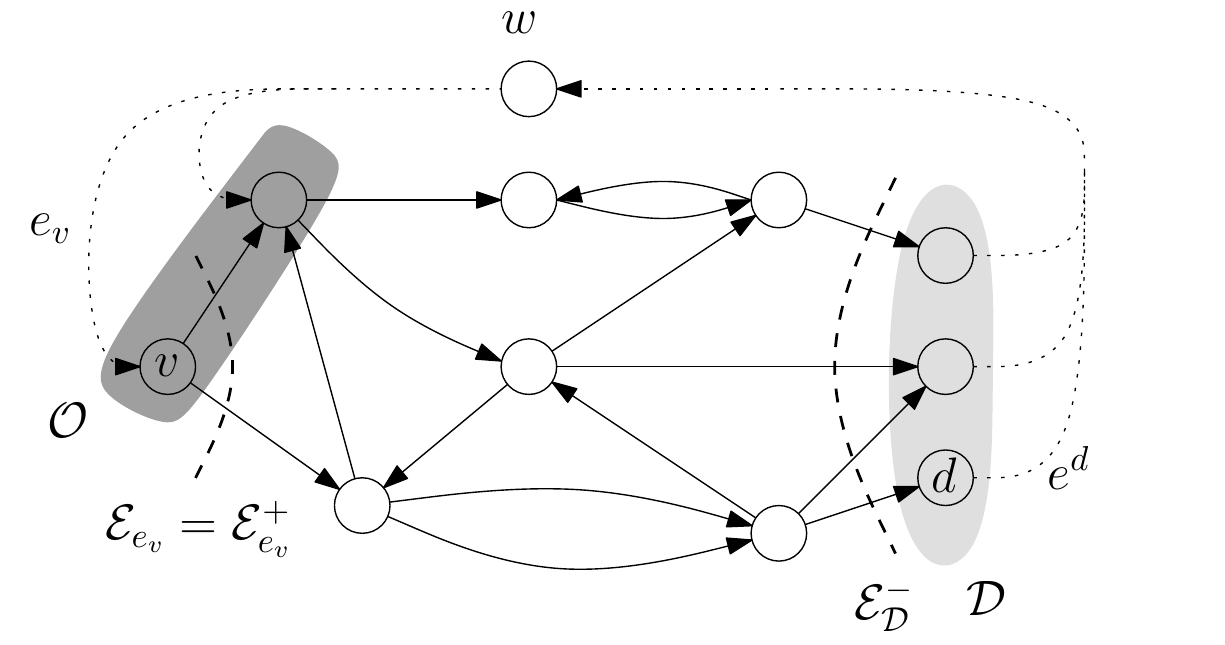}
	\caption{A network $\mc G$ and the augmented network $\mc G^a$. The links in $\E_\O^-$ and $\E_\D^+$, added in $\G^a$, are shown in dotted line. The set of origins and the set of destinations are shown in dark grey and light grey, respectively.}
	\label{fig:throughputOptimalNetwork}
\end{figure}

We now move on to introducing flow dynamics over $\G$. We consider autonomous dynamical systems of the form 
	\be
		\label{flowdynamics}
			\dot\rho=A^{\text{in}}F'(\rho)\onebf- A^{\text{out}}F(\rho)\onebf\,,
	\ee
where: $\rho(t)\in\RR^{\mc E}_+$ is a vector state whose $e$-th entry $\rho_e(t)$ represents the time-varying density on link $e$; $F(\rho)\in\RR_+^{(\E\cup\E_{\O}^-)\times (\E\cup\E_{\D}^+)}$ is the matrix of link-to-link flows with $F_{ij}(\rho)$ denoting the flow from link $i$ to link $j$; $A^{\text{in}}\in\RR^{\mc E\times(\E\cup\E_{\D}^-)}$ and $A^{\text{out}}\in\RR^{\mc E\times(\E\cup\E_{\O}^-)}$ are appropriate projection matrices from $\E\cup\E_{\D}^-$ and $\E\cup\E_{\O}^-$, respectively, onto $\mc E$; and $\onebf$ is the all-one vector (of the correct dimension). In order to match the topology and capacity constraints modeled by $\mc G$, the inflow $\lambda$, and invariance of the nonnegative orthant $\RR_+^{\E}$, it is assumed that: $F_{ij}(\rho)\equiv0$ if $\tau_i\ne\sigma_j$; 
$(F(\rho)\onebf)_{e_o}\equiv \lambda_o$ for all $o\in\mc O$;
and $F_{ij}(\rho)=0$ for all $j$ whenever $\rho_i=0$. 
We shall refer to \eqref{flowdynamics} as a \emph{dynamical flow network}.

To every link $e\in\mc E$ we associate a possibly finite buffer capacity $\rhomax_e\in(0,+\infty]$ and loosely use the phrase \emph{a set of links getting congested} to refer to the fact that the densities on those links approach their respective buffer capacities. 
We will focus on dynamics in $\R := \prod_{e\in\E}[0,B_e)$ and require $F(\rho)$ to be Lipschitz continuous on $\mc R$, so that standard analytical results imply, for every $\rho^{\circ}\in \mc R$, the existence and uniqueness of a solution $\{\rho(t):\,0<t<\kappa(\rho^{\circ})\}$ of \eqref{dynsysexample} starting from $\rho(0) = \rho^{\circ}$ which is well defined up to 
	$$    
		\kappa(\rho^{\circ}) :=   \sup\{t \ge 0:\, \rho(t)\in\R, \rho(0) = \rho^\circ\}	\,,
	$$
i.e., as long as $\rho(t)$ stays within $\mc R$. Note that, because of invariance of the nonnegative orthant, $\kappa(\rho^{\circ})$ coincides with the first time the solution of \eqref{flowdynamics} starting from $\rho(0)=\rho^{\circ}$ hits the buffer capacity on some link. 

We focus on flow dynamics that are distributed in the following sense: the flow from  $e\in\mc E\cup\E^-_{\O}$ to a downstream link $j\in\E^+_e\subseteq\E\cup\E_\D^+$ depends only on the local density vector 
	$$
		\rho^e := \{\rho_k:\,k\in\mc E_e\}\,,
	$$ 
where we recall that $\E_e=\E^+_e\cup\{e\}$ by \eqref{Eedef}. We will emphasize such functional dependence on local densities by writing the flow from $e\in\mc E\cup\E^-_{\O}$ to $j\in\E^+_e\subseteq\E\cup\E_\D^+$ as
	\be
		\label{distr}F_{ej}(\rho)=f_{e\to j}(\rho^e)\,,
	\ee
and referring to the family of flow functions $f=\{f_{e\to j}(\rho^e)\}$ as a \emph{distributed routing policy}. 
We will also use the notation 
	$$
		\fin_e(\rho) \defineas \sum_{j\in \E^-_e}f_{j\to e}(\rho^j)\,,\qquad 
		\fout_e(\rho^e):=\sum_{j \in \mc E_e^+} f_{e\to j}(\rho^e)
	$$  
for the total inflow and outflow, respectively, of a link $e \in \E$, so that \eqref{flowdynamics} reads
		\begin{equation}
        \label{equation:system}
            \dot{\rho}_e = \fin_e(\rho) - \fout_e(\rho)\,, \qquad e\in\E\,.
    \end{equation}  
Note that since $\sum_{e \in\E_v^-}\fout_e(\rho) = \sum_{e \in\E_v^+}\fin_e(\rho)$ for  $v \in \V$, \eqref{equation:system} and (\ref{flowdynamics}) imply mass conservation at the nodes.
}

Next, we formalize some fundamental properties of a class of distributed routing policy. 
As some of these characterize the behavior in the limit as some links get congested, we need to introduce the following notation: For $e\in(\E\cup\E_\O^-)\setminus\mc E^-_{\mc D}$, put $\ov\rho^e \defineas \{\rhomax_j:\,j\in \E_e\}$, and  let
	\be\label{Rbedef}
		\Rb_e \defineas 
			\begin{cases}
				\prod_{j\in\E_e}[0,B_j]\setminus\{\ov\rho^e\}, & \se \,\,e \in (\E\cup\E_\O^-)\setminus\mc E^-_{\mc D}	\\
				[0, B_e),&\se \,\,e\in\E_\D^-
			\end{cases}
	\ee
denote the set of possible densities on $e$ and links downstream to $e$ when not all of these links are congested (see Figure \ref{fig:throughputOptimalRe}). Finally, let the set of feasible outflows on the links downstream to $e$ under capacity constraint be defined as 
	$$
		\F_e \defineas  
			\begin{cases}
				\left\{x\in\RR_+^{\E^+_e}:\,\sum\nolimits_{j \in \mc E_e^+}x_j\leq \fmax_e \right\}, & \se \,\,e \in \E\cup\E_\O^-\setminus\mc E^-_{\mc D}	\\
				[0,C_e],&\se \,\,e\in\E_\D^-\,.
			\end{cases}
	$$
	
	\begin{figure}[t]
	\centering
		\includegraphics[height = 3cm]{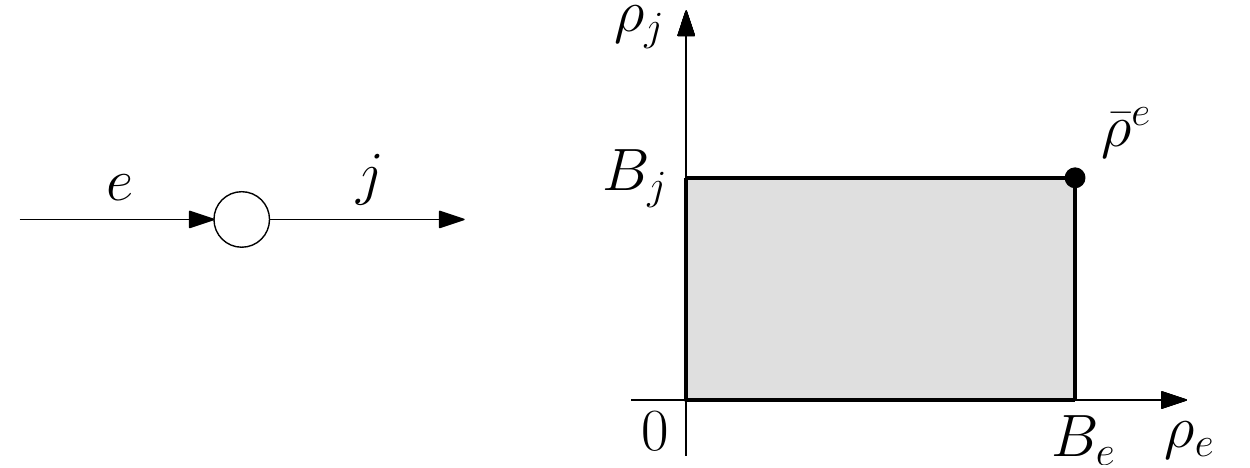}
	\caption{
The set $\Rb_e$ when $e\in\E$ has a unique downstream link $j\in\E$. It corresponds to the grey area with the solid line boundary, except the point $\bar{\rho}^e = \{\rhomax_e, \rhomax_j\}$ (represented as $\bullet$). 
	\label{fig:throughputOptimalRe}}
\end{figure}

\begin{definition}
\label{definition:distributedRoutingPolicy} 
Let $\mc G=(\mc V,\mc E,C)$ be a network satisfying Assumption \ref{assumption:connectivity} with vector of inflows $\lambda\in\RR_+^{\mc V\setminus\D}$ and buffer capacities $\{B_e\in(0,+\infty]:\,e\in\E\}$. A distributed routing policy $f$ is a family of Lipschitz-continuous maps 
    \begin{equation}
        \label{equation:definitionFlow}
        f^e: \Rb_e \to \F_e\,,\qquad     e \in \E\cup\E_\O^-\,,   
      \ee
such that
	$
		f^e(\rho^e) 
			= \left\{f_{e\to j}(\rho^e)\right\}_{j \in \E_e^+} 
	$
satisfy, at the origins
	\begin{equation}
  	\label{equation:definitionFlowOrigin}
    	\fout_{e_v}(\rho^{e_v}) \equiv \lambda_v, \qquad \forall v\in\O
	\end{equation}
and, for all $e \in \E$ and $\rho^e \in \Rb_e$,
	\begin{equation}
  	\label{equation:definitionFlowNoDensity}
    	\rho_e=0			\qquad\Longrightarrow\qquad \fout_e(\rho^e)=0\,,
	\end{equation}        
	\begin{equation}
		\label{equation:definitionFlowCongestedUpstream}
   		\rho_e = \rhomax_e \qquad\Longrightarrow\qquad \fout_e(\rho^e) = \fmax_e\,,
   \end{equation}
and, for all $e \in \left( \E\cup\E_\O^- \right) \setminus \mc E_{\mc D}^-$, $k \in \E_e^+$, $\rho^e \in \Rb_e$
	\begin{equation}
		\label{equation:definitionFlowCongestedDownstream}
    	\rho_k = \rhomax_k  \qquad\Longrightarrow\qquad f_{e\to k}(\rho^e) = 0\,.
   \end{equation}
\end{definition}

Observe that the domain of $f^e$ is $\Rb_e$, thus for $e\not\in\E_\D^-$ it is not defined at the point $\ov\rho^e = \{B_j:j\in\E_e\}$, where \eqref{equation:definitionFlowCongestedDownstream} and \eqref{equation:definitionFlowCongestedUpstream} cannot hold simultaneously. On the other hand, $f^e$ is well defined when at least one of the links around $e$ is not congested. Also, note that \eqref{equation:definitionFlowNoDensity} and \eqref{equation:definitionFlowCongestedDownstream} imply that $\fout_e(\rho^e) = 0$ if $\rho_e = 0$, i.e., there is no outflow from a link $e$ which is empty, or if $\rho_j = \rhomax_j$ for any $j \in \E_e^+$, i.e., if the densities on all the links outgoing from $\tau_{e}$ are at their buffer capacities.



We shall be interested in a special class of distributed routing policies, as per the following.
\begin{definition}
\label{definition:distributedRoutingPolicyMonotone}

A distributed routing policy $f$ is \emph{monotone} if, for all $e\in\E\cup\E_\O^-$, $\rho^{e}\in\Rb_{e}$, the functions $\{f^e\}$ satisfy
    \begin{align}
        &   \label{equation:monotone1}
        \frac{\partial f_{e \to j}}{\partial\rho_k}(\rho^e) \geq 0, &&
        \forall\, j\in\E_e^+, k \in \E_e\setminus\{j\}\,,\\
        &   \label{equation:monotone2}
        \frac{\partial}{\partial \rho_k}\fout_e(\rho^e) \leq 0, &&
        \forall\, k \in \E_e^+\,,
    \end{align}
for almost every $\rho^e\in\Rb_e$. A monotone distributed policy is \emph{strongly monotone} if, for all $e\in\E\cup\E_\O^-$, and almost every $\rho^{e}\in\Rb_{e}$, the inequalities in \eqref{equation:monotone1} and \eqref{equation:monotone2} are strict.
\end{definition}

Under monotone distributed routing policies, \eqref{dynsysexample} defines a \emph{cooperative} dynamical system (see \cite{HirschPS03}), since 
    \be
        \label{cooperative}
        \frac{\partial \fin_e}{\partial \rho_k}(\rho) \geq 0,\qquad \frac{\partial \fout_e}{\partial \rho_k}(\rho) \leq 0\qquad \forall e, k \in \E, e\neq k\,.
    \ee
Then, Kamke's theorem \cite[Th. 1.2]{HirschPS03} implies that \eqref{dynsysexample} is a monotone system \cite{HirschPS03}, i.e.,
    \be\label{eq:monotonicity}
        \rho(0) \leq \trho(0)\quad \Rightarrow\quad \rho(t) \leq \trho(t)\,,\quad \forall t\in[0,\kappa(\tilde{\rho}(0)))\,,
\ee
and thus clearly $\kappa(\rho^{\circ})\le\kappa(\zerobf)$ for all $\rho^{\circ}\in\mc R$. 

\newmaterial{
\begin{remark}
As shown in Lemma~\ref{corollary:contraction}, \eqref{dynsysexample} belongs to the class of compartmental systems, a class of monotone systems extensively used in the study of flow networks, such as transportation networks \cite{JacquezSIAMReview93}. It is also interesting to point out that, in the PDE literature, monotonicity is a property know to hold for entropy solutions of scalar conservation laws such as the traffic equation \cite[Proposition 2.3.6]{Serre:99}.  
\end{remark}

The monotonicity properties of the proposed policies describe both the behavior the particles in the network and the effect of flow control. In particular, \eqref{equation:monotone1} describes the fact that while particles might have preferred paths, they tend to deviate to avoid congested links, i.e., the higher $\rho_k$, the less the flow towards $k$. Instead, \eqref{equation:monotone2} requires that when density is increasing downstream of a link, the total flow from the link should not increase. We notice that this allows these policies to implicitly \emph{back-propagate}, towards the origins, the information that some branches of the network are getting congested. 
}

We conclude this section with an example of monotone distributed routing.

\begin{example}
\label{example:DistributedMonotonePolicy}
For every link $e\in\mc E$, let $\varphi_e:[0,\rhomax_e)\to[0,+\infty)$ be Lispchitz continuous, strictly increasing, and such that $\varphi_e(0)=0$ and $\lim_{\rho_e\uparrow B_e}\varphi_e(\rhomax_e)=+\infty$. E.g., for $\beta_e>0$, $\varphi_e(\rho_e)=\beta_e \rho_e /(B_e-\rho_e)$ if $\rhomax_e < +\infty$, or $\varphi_e(\rho_e)=\beta_e \rho_e$ if $\rhomax_e=+\infty$.
Define
	$$
		f_{e\to j}(\rho^e)
			=
				\begin{cases}
				 \ds \fmax_e \left(1-\gamma_e \right) \gamma_j/Z& \text{ if } e\in\E\setminus\E_\D^-\,,\\ 
				\ds \fmax_e \left(1-\gamma_e\right)& \text{ if } e \in\E_d^-, d\in\D,j=e^d\,, \vspace{0.02in}	\\			
				\ds\lambda_v \gamma_j/ Z& \text{ if } e = e_v, v \in\O, j \in \mc E_{e_v}\,,
				\end{cases}
	$$
where $\gamma_i:=\exp(-\varphi_i(\rho_i))$ and $Z:=\sum_{k\in\E_e} \gamma_k $. Observe that $f_{e\to j}(\rho^e)$ is defined for  
$\rho^e\in\R_e$ where 
	$
		\R_e:= \prod_{j\in\E_e}[0,\rhomax_j)
	$ if $e\notin\E$ and $\R_e=[0,B_e)$ for $e\in\E_{\D}^-$ 
and can be extended by continuity to $\Rb_e$ (as defined in (\ref{Rbedef})), but not to the point $\ov\rho^e$. 
Then, it can be readily verified that this defines a strongly monotone distributed routing policy.
\end{example}

\section{Main results}
\label{section:MainResults}
In this section, we present the main contributions of the paper. The first result is Theorem \ref{theorem:mainResult}, which states a dichotomy. If the inflow is less than the capacity of every cut, then there exists a globally asymptotically stable equilibrium density $\rho^* \in \R$.
Otherwise, the network is divided in two parts by a cut $\mc S$, such that the densities on the links in $\mc E_{\mc S}^+$ approach their buffer capacities simultaneously. 

\begin{theorem}
\label{theorem:mainResult} 
Let $\mc G=(\mc V,\mc E,C)$ be a network satisfying Assumption \ref{assumption:connectivity} with vector of inflows $\lambda \in \RR_+^{\mc V \setminus \mc D}$, and $f$ be a monotone distributed routing policy. For $\rho^\circ \in \R$, let $\{\rho(t):\,0\le t<\kappa(\rho^{\circ})\}$ be the solution of the dynamical flow network \eqref{dynsysexample} with initial condition $\rho(0) = \rho^\circ$. 
Then,
    \begin{enumerate}
        \item[(i)] if $\max_{\U} \left\{\lambda_{\U}-C_{\U} \right\}<0$, then $\kappa(\rho^\circ) =+\infty$ for every initial density $\rho^{\circ}\in\mc R$; moreover, if $f$ is strongly monotone, then there exists an equilibrium density $\rho^* \in \R$ such that $\lim_{t\to\infty}\rho(t) = \rho^*$ for every initial density vector $\rho^{\circ}\in\mc R$.
        \item[(ii)] if $\max_{\U}\left\{\lambda_{\U}-C_{\U}\right\}>0$, or if $\max_{\U}\left\{\lambda_{\U}-C_{\U}\right\}=0$ and if $f$ is strongly monotone, then, for every initial density $\rho^{\circ}\in\mc R$, there exists a cut $\S$ such that \be\label{maintheo:claim2}\lim_{t\to\kappa(\rho^\circ)}\rho_e(t) = \rhomax_e,\quad \forall e \in \E^+_{\S}\,. \ee   						
    \end{enumerate}
\end{theorem}

\newmaterial{
Theorem~\ref{theorem:mainResult}, together with $\max_{\U} \left\{\lambda_{\U}-C_{\U} \right\}<0$ being a necessary condition for the network to admit an equilibrium,  implies that monotone distributed policies are maximally stabilizing. In terms of resilience, Theorem~\ref{theorem:mainResult} reads $\nu(0) = \lambda_{\V\setminus\D} - \fmax_\G$, i.e., throughput loss only occurs if the capacity is reduced in such a way that the min-cut capacity constraint is violated.

\begin{remark}
This framework can easily be applied to scenarios where nodes have maximum outflow capacity $C_v$ and/or finite buffer capacity $B_v$ to \emph{store-and-forward} particles. In order to bring this setup within the purview of Theorem~\ref{theorem:mainResult}, one can replace every node $v \in \mc V \setminus \mc D$ with a pair of nodes $v_1$ and $v_2$, which inherit incoming and outgoing links, respectively, from node $v$, and are connected by a directed link $(v_1,v_2)$ with flow and buffer capacities equal to $C_v$ and $B_v$, respectively. On the other hand, for a destination node $d \in \mc D$, we assign the buffer and outflow capacities to the link $(d,w)$. 
From an implementation perspective, this construction allows one to interpret the routing at node $v$ as the combination of routing at $v_1$ and $v_2$ with the buffer of link $(v_1,v_2)$ serving as the internal state.  
\end{remark}

\begin{remark}
Theorem~\ref{theorem:mainResult} can be extended to time-varying inflows $\lambda(t)$. In particular, input-output monotonicity~\cite{AngeliTAC03} implies that the solutions of \eqref{dynsysexample} with time-varying inflows $\lambda_v(t)$ and constant inflows $\tilde{\lambda}_v := \sup_{t \geq 0} \lambda_v(t)$, $v \in \mc V \setminus \mc D$, respectively, satisfy $\rho(t) \leq \tilde{\rho}(t)$ when started from initial conditions $\rho(0)\leq \tilde\rho(0)$. Then, it follows from Theorem~\ref{theorem:mainResult} that, if $\max_{\U} \{\tilde\lambda_{\U}-C_{\U} \}<0$, then $\limsup\rho(t)\leq\lim\tilde{\rho}(t)=\tilde{\rho}^*$ as $t \to \infty$ under strongly monotone distributed routing. 

For the infinite buffer capacity case, a stronger result holds. Let $\hat{\lambda}_v:=\limsup \frac{1}{t} \int_0^{t} \lambda_v(t) \, dt$. It is then possible to show that  $\max_{\mc U}\{\hat{\lambda}_{\mc U} - C_{\mc U} \} < 0$ implies that every trajectory remains bounded in time under monotone distributed routing.
\end{remark}
}

\subsection{Overload behavior with finite buffer capacities}

The following proposition gives a more detailed characterization of what happens when the capacity constraints are violated in the case of finite buffer capacities.

\begin{proposition}
\label{corollary:FiniteBuffer}
Let $\mc G=(\mc V,\mc E,C)$ be a network satisfying Assumption \ref{assumption:connectivity} with vector of inflows $\lambda \in \RR_+^{\mc V \setminus \mc D}$ and finite buffer capacities $B_e\in(0,+\infty)$, $e\in\E$, and $f$ be a monotone distributed routing policy. Assume that $\max_{\mc U}\left\{\lambda_{\U}-C_{\U}\right\}>0\,.$
Then, for every $\rho^\circ \in \R$, 
\be\kappa(\rho^{\circ})\le\min_{\mc U:\,\lambda_{\U}>C_{\mc U}}\frac{\sum_{e\in\E_{\mc U}^+}\left(B_e-\rho^\circ_e\right)}{\lambda_{\U}-C_{\U}}\,,
\label{kappaUB}\ee
and there exists a cut $\mc S$, possibly depending on $\rho^{\circ}$, such that $\lambda_{\S}>C_{\S}$ and
\begin{equation}
\label{Sexplodes}
\begin{split}
& \rho_e(t)<B_e\,,\, \forall e\in\E\,,\ 0\le t<\kappa(\rho^{\circ})\,,\\ & \lim_{t\to\kappa(\rho^\circ)}\rho_e(t) = \rhomax_e,\quad \forall e \in \E_{\S}^+\,,\end{split}
\end{equation}
where $\{\rho(t):\,0\le t<\kappa(\rho^{\circ})\}$ is the solution of the dynamical flow network \eqref{dynsysexample} with initial condition $\rho(0)=\rho^\circ$.
\end{proposition}

Proposition \ref{corollary:FiniteBuffer} states that, if the buffer capacities are finite and some cut constraints are violated, then, for every initial density $\rho^\circ$, all the links in $\E_{\S}^+$, where $\mc S$ is a cut such that $\lambda_\S>C_\S$, will reach their buffer capacities simultaneously at time $\kappa(\rho^\circ)$. We notice that, when there are multiple cuts violating the capacity constraint, then the cut $\S$ may depend on the initial condition $\rho^\circ$. The dependence on the initial density $\rho^\circ$ is also evident in \eqref{kappaUB}. While it may be tempting to identify the cut $\U$ minimizing the right hand side of \eqref{kappaUB} with the cut $\S$ of \eqref{Sexplodes}, it is worth stressing that \eqref{kappaUB} is merely an upper bound on $\kappa(\rho^\circ)$. In fact, in contrast to the right-hand side of \eqref{kappaUB}, the cut $\S$ of \eqref{Sexplodes} may depend on finer details of the routing policy, rather than just its inflow and buffer capacities. 


\subsection{Overload behavior with infinite buffer capacities}
The following result, similar to Proposition \ref{corollary:FiniteBuffer}, characterizes the way congestion occurs in case of infinite buffer capacities.
\begin{proposition}
\label{proposition:InfiniteBuffer}
Let $\mc G=(\mc V,\mc E,C)$ be a network satisfying Assumption \ref{assumption:connectivity} with vector of inflows $\lambda \in \RR_+^{\mc V \setminus \mc D}$ and buffer capacities $B_e=+\infty$, $e\in\E$. Let $f$ be a strongly monotone distributed routing policy. Assume that $\max_{\mc U}\left\{\lambda_{\U}-C_{\U}\right\}\ge0$.
Let 
\be\label{U*def}\mc U^*:=\bigcup\nolimits_{\mc U\in\mc M}\mc U\,,\qquad \mc M:=\argmax_{\U}\left\{\lambda_{\U}-C_{\U}\right\}\,.\ee
Then, for every $\rho^\circ \in \R$, the solution $\rho(t)$ of the dynamical flow network \eqref{dynsysexample} with initial condition $\rho(0) = \rho^\circ \in \R$ is such that $\kappa(\rho^\circ)=+\infty$ and
\be\label{limE*}
\begin{split}
& \lim_{t \to + \infty}\rho_e(t) = +\infty\,,\qquad \forall e\in\E_{\U^*}^+\,,\\ & \lim_{t \to + \infty} \frac{1}{t}\sum_{e\in\E_{\U^*}^+}\rho_e(t) =  \lambda_{\U^*} - C_{\U^*}\,.
\end{split}
\ee
Moreover, there exist $\rho_e^*\in [0,+\infty)$,  $e\in\E\setminus(\mc E^+_{\U^*}\cup\partial_{\U^*}^-)$, such that 
\be\label{limE*S}\lim_{t\to +\infty}\rho_e(t) = \rho_e^*\,,\qquad\forall e\in\E\setminus(\mc E^+_{\U^*}\cup\partial_{\U^*}^-)\,,\ee 
for every initial density $\rho^\circ \in \R$.
\end{proposition}

Proposition \ref{proposition:InfiniteBuffer} implies that, for infinite buffer capacities on all the links, there exists a cut $\U^*$, independent of initial condition $\rho^\circ$, such that, asymptotically, all the links in $\E_{\U^*}^+$ get congested. This is to be contrasted with the finite buffer capacity case, where the cut depends on the initial condition $\rho^\circ$. In addition, by \eqref{limE*S}, the densities on the links which do not get congested approach a unique limit point, and by \eqref{limE*} the total density grows linearly in time.
\newmaterial{In particular, the growth rate corresponds to the throughput loss in the network. As such, a throughput loss equal to $\delta$ is obtained by perturbing the network in such a way that the min-cut capacity of the perturbed network is $\tilde{\fmax}_\G = \lambda_{\V\setminus\D} - \delta$. Therefore, in terms of resilience, Proposition~\ref{proposition:InfiniteBuffer} yields $\nu(\delta) = \fmax_\G - \tilde{\fmax}_\G = \fmax_\G - \lambda_{\V\setminus\D} + \delta$, which is the maximum possible.} 
A comparison is due with \cite{ShahQS11}, which studies an acyclic queuing network with set of queues $\Q$ employing max-weight algorithm. It is shown that if $q(t) \in \RR_+^\Q$ is the vector of queue lengths, then $q(t)/t \to \hat{q}$ where $\hat{q}\in \RR_+^\Q$ is the solution to an optimization problem related to the parameters of the max-weight algorithm.

\section{Proofs}
\label{section:proof}

In this section we provide an $l_1$-contraction principle for monotone dynamical systems under conservation laws and prove that it applies to \eqref{flowdynamics}. We then characterize the behavior of dynamical flow networks when the vector of densities admits a limit point. Finally, we prove the main results.

\subsection{$l_1$-contraction principle for monotone conservation laws}

We state and prove an $l_1$-contraction principle for a class of monotone dynamical systems under conservation laws, which includes system \eqref{dynsysexample} under monotone distributed routing policy. As such, it will be instrumental in proving existence and stability of equilibria for dynamical flow networks.

\begin{lemma}
\label{lemma:l1Contraction} For a non-empty closed hyper--rectangle $\Omega\subseteq\RR^n$, let $g:\Omega\to\RR^n$ be
Lipschitz and such that
    \be
    \label{equation:contractionAssumption1}
        \frac{\partial}{\partial x_j}g_i(x)\ge0\,,\qquad \forall\,i\ne j\in\{1,\ldots,n\}
    \ee
    \be
    \label{equation:contractionAssumption2}
            \sum_{1\le i\le n}\frac{\partial}{\partial x_j}g_i(x)\le0\,,\qquad \forall\,j \in\{1,\ldots,n\}
    \ee
 for almost every $x\in\Omega$.
Then \be \!\!\!\!\label{ineq}\sum_{1\le i\le n}\sgn{x_i-y_i}\left(g_i(x)-g_i(y)\right) \leq 0,\quad x, y \in
\Omega\,.\ee

Moreover, if
    \begin{enumerate}
        \item[(i)] there exists some $ j\in\{1, \dots, n\}$ such that the inequality \eqref{equation:contractionAssumption2}
        is strict for almost all $x \in \Omega$,
    \end{enumerate}
then inequality \eqref{ineq} is strict for all $x, y\in\Omega$ such that $x_j\ne y_j$.

If
    \begin{enumerate}
        \item[(ii)] for every proper subset $\mc K\subseteq\{1,\ldots,n\}$, there exist $i\in\mc K$, and
            $j\in\{1,\ldots,n\}\setminus\mc K$ such that inequality \eqref{equation:contractionAssumption1} is strict for almost all $x\in\Omega$,
    \end{enumerate}
then inequality \eqref{ineq} is strict for all $x\ne y$ such that $x\not< y$ and $y \not< x$.

Finally, if (i) and (ii) hold true, then inequality \eqref{ineq} is strict for all $x, y\in\Omega$ such that $x \neq y$.
\end{lemma}
\begin{proof}
First note that, according to Rademacher's theorem, e.g., see \cite{evans1992measure}, Lipschitz continuity implies differentiability almost everywhere. 
For $\mc A\subseteq\{1,\ldots,n\}$, put $\mc A^c:=\{1,\ldots,n\}\setminus \mc A$, and $g_{\mc
A}(z):=\sum_{a\in\mc A}g_a(z)$. Fix some $x,y\in\Omega$, and put $ \mc I = \{i: x_i > y_i\}$, $\mc J = \{i: x_i < y_i\}
$. Let $\xi\in\Omega$ be such that $\xi_i = x_i$ for $ i \in \mc I$ and $\xi_i=y_i$ for $i \in \mc I^c$. Consider the
segments $\gamma_{\mc I}$ from $y$ to $\xi$ and $\gamma_{\mc J}$ from $x$ to $\xi$. For $\mc A\subseteq\{1,\ldots,n\}$,
and $\mc B\in\{\mc I,\mc J\}$, define the path integral
    $$
        \Gamma^{\mc A}_{\mc B}:=\int_{\gamma_{\mc B}}\nabla g_{\mc A}(z)\cdot \de z\,.
    $$
Then, \eqref{equation:contractionAssumption2} implies that
    \begin{align}
      	g_{\mc I}(x)-g_{\mc I}(y)
        	&	=\ds  \Gamma^{\mc I}_{\mc I}-\Gamma^{\mc I}_{\mc J} \le -\Gamma^{\mc I^c}_{\mc I} -\Gamma^{\mc I}_{\mc J} \label{ineq1}
        		\\
       g_{\mc J}(x)-g_{\mc J}(y) 
       		&	= \Gamma^{\mc J}_{\mc I}-\Gamma^{\mc J}_{\mc J} \ge  \Gamma^{\mc J}_{\mc I}+\Gamma^{\mc J^c}_{\mc J}\,.\label{ineq2}
    \end{align}

\newmaterial{
Denoting  $s_i:=\sgn{x_i-y_i}$, the definition of $\I$ and $\mc J$, and \eqref{ineq1} and \eqref{ineq2}, yield}
	\begin{align*}
		\sum_{i}s_i\left(g_i(x)-g_i(y)\right)	
			&	=	g_{\mc I}(x)-g_{\mc I}(y)-g_{\mc J}(x)+g_{\mc J}(y)	\\
			& \le -\Gamma^{\mc I^c}_{\mc I}-\Gamma^{\mc I}_{\mc J}-\Gamma^{\mc J}_{\mc I}-\Gamma^{\mc J^c}_{\mc J}\,.
	\end{align*}
	
Observe that, by \eqref{equation:contractionAssumption1}, $\mc A\cap\mc B=\emptyset$ implies $\Gamma^{\mc
A}_{\mc B}\ge0\,$, so that \eqref{ineq} follows immediately.

Notice that, if there exists some $ j\in\{1, \dots, n\}$ such that inequality \eqref{equation:contractionAssumption2}
is strict for almost every $x \in \Omega$, and $x_j>y_j$ ($x_j<y_j$), then \eqref{ineq1} (respectively, \eqref{ineq2})
is a strict inequality, hence so is \eqref{ineq}, thus proving the second claim.

Now, assume that $x \neq y$, $x\not< y$ and $y \not< x$. Then, it follows from the definition of the sets $\mc I$ and $\mc J$ that the sets $\mc I^c$ and $\mc J^c$ are non-empty. We also have that $\mc I^c \cap \mc J^c = \setdef{i \in \until{n}}{x_i=y_i}$. Since $x \neq y$, this implies that $\mc I^c \cap \mc J^c \neq \until{n}$. Therefore, at least one of $\mc I^c$ and $\mc J^c$ is a proper subset of $\until{n}$. If say $\mc I^c$ is a proper subset, then the condition in (ii) in the statement of the lemma implies that \eqref{equation:contractionAssumption1} is strict for some $i \in \mc I$ and $j \in \mc I^c$. Therefore, $\Gamma^{\I^c}_{\I}>0$, and the third claim follows.

Finally, the last claim is implied by the previous two: if $x < y$ or $y < x$, then trivially $x_j \neq y_j$ for all $j \in \until{n}$ and the strict inequality in \eqref{ineq} follows from the claim associated with condition (i); if $x \not< y$ and $y \not< x$, the strict inequality in \eqref{ineq} follows from the claim associated with  condition (ii).
\end{proof}

Lemma~\ref{lemma:l1Contraction} implies the following $l_1$-contraction principle for dynamic networks with monotone distributed routing policies.

\begin{lemma}
\label{corollary:contraction}
Let $\mc G=(\mc V,\mc E,C)$ be a network satisfying Assumption \ref{assumption:connectivity}, $f$ be a monotone distributed routing policy, and $\hat\rho^\circ,\tilde\rho^\circ \in \R$. 
Let $\hat\rho(t)$ and $\tilde\rho(t)$ be the solutions to the system \eqref{dynsysexample} with initial conditions $\hat\rho(0)=\hat\rho^{\circ}$, and $\tilde\rho(0)=\tilde\rho^{\circ}$, respectively. 
Define $\varphi(t) := ||\hat\rho(t) - \tilde\rho(t)||_1$ for $0\leq t < \min\{\kappa(\hat\rho^\circ),\kappa(\tilde\rho^\circ)\}$. Then $\dot{\varphi}(t) \leq 0$. Moreover, if the routing policy is strongly monotone, then $\dot{\varphi}(t) = 0$ if and only if $\hat\rho(t) = \tilde\rho(t)$.
\end{lemma}
\begin{proof}
It is easily verified that the properties of monotone distributed routing policies \eqref{equation:monotone1} and \eqref{equation:monotone2} imply \eqref{equation:contractionAssumption1} and \eqref{equation:contractionAssumption2} for the function $g_e(\rho):=\fin_e(\rho)-\fout_e(\rho)$. Therefore, the first claim in Lemma~\ref{lemma:l1Contraction} gives 
$$\dot{\varphi}(t) = \sum_{e}\sgn{\hat\rho_{e}(t) - \tilde\rho_{e}(t)}(g_e(\hat\rho(t)) - g_e(\tilde\rho(t)))\le0\, $$
if the distributed routing policy is monotone. 

We now show that conditions (i) and (ii) in Lemma~\ref{lemma:l1Contraction} follow from the strong monotonicity property of the distributed routing policies. To that effect, for any $j \in \mc E_{\mc D}^-$, we have that
	\begin{align*}		
		\frac{\partial}{\partial \rho_j}\sum_{e \in \mc E} g_e(\rho)
			&=	\frac{\partial}{\partial \rho_j} \Bigg(\sum_{v \in \O} \lambda_v - \sum_{i \in \mc E_{\mc D}^-} \fout_i(\rho^i)  \Bigg)
				\\
			&	=	-\frac{\partial}{\partial \rho_j}  \fout_j(\rho^j) < 0
	\end{align*}
where the strict inequality follows from the strict version of \eqref{equation:monotone1} characterizing strongly monotone routing policies. This establishes condition (i) in Lemma~\ref{lemma:l1Contraction}. In order to connect condition (ii) in Lemma~\ref{lemma:l1Contraction}, consider any proper subset $\mc K \subsetneq \mc E$. It is easily seen that there exist $i \in \mc K$ and $j \in \mc K^c$ such that: either (a) $\tau_j=\sigma_i$ or $\sigma_j=\sigma_i$; or (b) $j \in \mc E_i^+\cap\E$. In case (a), 
	$$
		\frac{\partial g_i }{\partial \rho_j}(\rho)
			= \frac{\partial\fin_i}{\partial \rho_j}  (\rho)  
			= \sum_{e \in \mc E_i^-} \frac{\partial f_{e \to i}}{\partial \rho_j} (\rho^e) 
				> 0,
	$$
where the strict inequality follows from the strict version of \eqref{equation:monotone1} that holds true for a strongly monotone routing policy. In case (b), $\frac{\partial}{\partial \rho_j} g_i(\rho) = - \frac{\partial}{\partial \rho_j}  \fout_i(\rho^i) > 0,$ where the strict inequality follows from the strict version of \eqref{equation:monotone2} that holds true for a strongly monotone routing policy. The last claim in Lemma~\ref{corollary:contraction} follows now from the last claim in Lemma~\ref{lemma:l1Contraction}. 
\end{proof}

\subsection{Properties of limit density vectors}

For an initial density $\rho^{\circ}\in\mc R$, let us consider the following subsets of $\mc E$:
	\be
  	\label{BWCF}
    	\ba{ll}
				\ds 	\mc B := \left\{\lim\rho_e(t)=B_e\right\},									&  	\!\!\!\!\mc W:=\left\{\limsup\rho_e(t)<B_e\right\}, 								\\
				[7pt]	\mc Z_{\text{o}}:=\left\{\lim\fout_e(\rho^e(t))=0\right\}\!,	& 	\!\!\!\!\mc Z_{\text{i}}:=\left\{\lim\fin_e(\rho^e(t))=0\right\} ,	\\
				[7pt] \mc C:=\left\{\lim\fout_e(\rho^e(t))=C_e\right\}\!, 					& 	\!\!\!\!\mc Z:= \mc Z_{\text{i}}\cup\mc Z_{\text{o}} ,
        \ea
    \ee	
where the limits are meant as $t\uparrow\kappa(\rho^{\circ})$ and the curly brackets are meant as defining the sets of those links $e$ such that the enclosed condition is satisfied.

Observe that the definitions in \eqref{BWCF} do not assume existence of a limit density. However, if a limit $\rho^*=\lim_{t\uparrow\kappa(\rho^{\circ})}\rho(t)$ exists, then $\mc E=\mc B\cup\mc W$. Also, in general, existence of the limit density $\rho^*$ does not necessarily imply existence of the limit outflow $\lim_{t\uparrow\kappa(\rho^\circ)}\fout_e(\rho^e(t))$ or the limit inflow $\lim_{t\uparrow\kappa(\rho^\circ)}\fin_e(\rho^e(t))$ for every $e\in\E$. 
Finally, observe that $\C\cap\Z_{\text o} = \emptyset$, and that $\B\cap\C\cap\Z_{\text i} = \emptyset$, since $\lim_{t\uparrow\kappa(\rho^{\circ})}\dot\rho_e(t)=-C_e<0$ for all $e\in\C\cap\Z_{\text{i}}$, which is incompatible with $e\in\B$. 

The following lemma characterizes the behavior of $\rho(t)$ starting from some $\rho(0)=\rho^{\circ}\in\R$, as $t$ approaches $\kappa(\rho^\circ)$. 

\begin{lemma}
\label{lemma:propertiesWBCF} 
Let $\mc G=(\mc V,\mc E,C)$ be a network satisfying Assumption \ref{assumption:connectivity}, and $f$ be a monotone distributed routing policy. Let $\rho^\circ \in \R$ be such that the solution $\rho(t)$ of the dynamical flow network \eqref{dynsysexample} with initial condition $\rho(0)=\rho^{\circ}$ admits a limit $\rho^*=\lim_{t\uparrow\kappa(\rho^{\circ})}\rho(t)$. 
Let $\B,\W,\C,\Z\subseteq\E$ be defined as in \eqref{BWCF}. Then,
    \begin{enumerate}
        \item\label{itempropertyAfterExploded} if $e \in \B$, then $e \in \C$, or $e\notin\mc E_{\mc D}^-$ and $\E_e^+ \subseteq \B$;
        \item\label{itemproperyBeforeExploded} if $e \in \B$, then $e \in \Z_{\text{i}}$, or $\E_{\sigma_e}^+ \subseteq\B$;
        \item\label{itempropertyOnlyBAfterNode} 
        if  $e\in\mc W\setminus\E_\D^-$ and $\E_{e}^+ \subseteq \B$, then 
        $e\in\mc Z_{\text{o}}$.
    \end{enumerate}
\end{lemma}
\begin{proof}
1) First consider the case $e\in\mc E^-_{\mc D}$. Then, \eqref{equation:definitionFlowCongestedUpstream} implies that, if $e\in\mc B$, then $e\in\mc C$. On the other hand, assume that $e\notin\mc E^-_{\mc D}$. Then, if $e\in\mc B$ and $\E^+_e\nsubseteq\B$, necessarily $\{\rho^{*}_e\}_{e\in\mc E_e} \in \Rb_e$, so that property \eqref{equation:definitionFlowCongestedUpstream} implies that $e\in\mc C$. \\
2) Let $e$ be such that $\E_{\sigma_e}^+\not\subseteq\B$. Then,  property \eqref{equation:definitionFlowCongestedDownstream} implies that $\lim_{t\uparrow\kappa(\rho^{\circ})}\fin_e(\rho(t)) =0\,.$  \\
3) If $e\in\mc W\setminus\E^-_{\D}$ and $\E_{e}^+ \subseteq \B$, then property \eqref{equation:definitionFlowCongestedDownstream} implies that $\lim_{t\uparrow\kappa(\rho^{\circ})}\fout_e(\rho(t)) =0$. 
\end{proof}

The following fundamental result states that either $\B = \emptyset$, or there exists a cut on the origin side of which the densities hit the buffer capacities.

\begin{lemma}
\label{lemma:existenceOfCut} 
Let $\G=(\V, \E,C)$ be a network satisfying Assumption \ref{assumption:connectivity}, and $f$ be a monotone distributed routing policy with vector of inflows $\lambda$. Let $\rho^\circ \in \R$ be such that the solution $\rho(t)$ of the dynamical flow network \eqref{dynsysexample} with initial condition $\rho(0)=\rho^{\circ}$ admits a limit $\rho^*=\lim_{t\uparrow\kappa(\rho^{\circ})}\rho(t)$. 
Let $\mc B,\mc W,\mc C,\mc Z\subseteq\mc E$ be defined as in \eqref{BWCF}. Then, either $\E = \W$, or there exists a cut $\S$ with $C_{\S}\le\lambda_{\S}$  such that $\mc E_{\mc S}^+ \subseteq \B$, $\partial_\S^+ \subseteq \C$, $\partial_\S^- \subseteq \Z$, and $\E\setminus(\E^+_{\S}\cup\partial_\S^-)\subseteq\W$.
\end{lemma}
\begin{proof}
Existence of the limit density $\rho^*$ implies that $\mc E = \mc B \cup \mc W$. Assume that $\mc E \neq \mc W$, and hence $\mc B \neq \emptyset$. Let $\S := \{v\in \V\setminus\mc D: \E_v^+\subseteq \B\}$. To start with, we prove that $\S \neq \emptyset$. To see this, consider a link $e \in \B$. If also $e\in\mc E_{\mc D}^-$, then statement \ref{itempropertyAfterExploded} of Lemma~\ref{lemma:propertiesWBCF} implies that $e \in \mc C$, and hence $e \notin \Z_i$. 
This combined with statement \ref{itemproperyBeforeExploded} of Lemma~\ref{lemma:propertiesWBCF} implies that $\mc E_{\sigma_e}^+ \subseteq \mc B$, and hence $\sigma_e \in \mc S \neq \emptyset$. On the other hand, if $e \in\mc B\setminus \mc E_{\mc D}^-$, then statement~\ref{itempropertyAfterExploded} of Lemma~\ref{lemma:propertiesWBCF} implies that $\mc E_e^+ \subseteq \mc B$ or $e \in \mc C$. In the former case, $\tau_e \in \mc S \neq\emptyset$. In the latter case, $e \in\C\cap\B$ implies again $e \notin \Z_i$, so that, statement \ref{itemproperyBeforeExploded} of Lemma~\ref{lemma:propertiesWBCF} yields $\mc E_{\sigma_e}^+ \subseteq \mc B$, hence $\sigma_e \in \mc S\ne\emptyset$. Hence, $\mc S\ne\emptyset$ and, since $\mc S\cap\mc D=\emptyset$ by construction, $\mc S$ is a cut. Also, by construction, $\mc E^+_{\mc S}\subseteq\mc B$. 

We prove now that $\partial_\S^+ \subseteq \C$. In fact, if $e \in \partial_\S^+$, then $\mc E_e^+ \not \subseteq \mc B$ for otherwise one would have $\tau_e \in \mc S$ so that $e \notin \partial_\S^+$. Hence $e \in \partial_\S^+$ implies $\{\rho^{*}_e\}_{e\in\E_e} \in \Rb_e$, which combined with \eqref{equation:definitionFlowCongestedUpstream} implies $e\in\mc C$.

On the other hand, for every $e \in \partial_\S^-$, one has $\E_{\sigma_e}^+\nsubseteq \B$ (since $\sigma_e\notin\S$) and $\E_e^+ \subseteq \B$ (since $\tau_e\in\mc S$). Therefore, 
 statement \ref{itemproperyBeforeExploded} of Lemma~\ref{lemma:propertiesWBCF} implies that $\partial_\S^-\cap\B \subseteq \Z_{\text{i}}$, 
while statement \ref{itempropertyOnlyBAfterNode} of Lemma \ref{lemma:propertiesWBCF} implies that $\partial_\S^-\cap\W \subseteq \Z_{\text{o}}$.

To show that $\E\setminus(\E^+_{\S}\cup\partial_\S^-)\subseteq\W$, it is sufficient to prove that, for every $e\in\B$ with $\sigma_e\notin\S$, necessarily $\tau_e\in\S$, so that $e\in\partial_\S^-$. Indeed, it follows from statement \ref{itemproperyBeforeExploded} of Lemma \ref{lemma:propertiesWBCF} that $e\in\B$ and $\sigma_e\notin\S$ (i.e., $\E^+_{\sigma_e}\nsubseteq\B$) imply that $e\in\Z_{\text{i}}$, so that $e\notin\C$ and statement \ref{itempropertyAfterExploded} of Lemma \ref{lemma:propertiesWBCF} implies that $\tau_e\in\S$.


{Finally, it follows from $\mc E_{\mc S}^+ \subseteq \B$ and $\E\setminus(\E^+_{\S}\cup\partial_\S^-)\subseteq\W$ that $\mc B=\E_{\S}^+\cup\partial_{\mc S}^-\cap\B$. Then, using $\partial_\S^+ \subseteq \C$, $\partial_\S^-\cap\B \subseteq \Z_{\text{i}}$, and $\partial_\S^-\cap\W \subseteq \Z_{\text{o}}$, one gets that
	\begin{multline*}
		\sum_{e\in\B}\dot{\rho}_e(t)= \lambda_\S + \!\!\!\sum_{e\in\partial_\S^-\cap\mc W}\!\!\!\!\fout_e(t)+ \!\!\!\sum_{e\in\partial_\S^-\cap\mc B}\!\!\!\fin_e(t) - \!\!\!\sum_{e\in\partial_\S^+}\!\fout_e(t)\\
		\stackrel{t\uparrow\kappa(\rho^{\circ})}{\longrightarrow} \lambda_\S - C_\S\,,
	\end{multline*}

Since $\rho_e(t)<B_e$ for  $t\in[0,\kappa(\rho^{\circ}))$ and $\lim_{t\uparrow\kappa(\rho^{\circ})}\rho_e(t)=B_e$ for all $e\in\B$, the above implies that  $\lambda_\S - C_\S\ge0$.}
\end{proof}

\subsection{Proof of Theorem \ref{theorem:mainResult}}\label{sec:proofmaintheo}
The results in the previous subsection assume existence of a limit density, which, in principle, is not guaranteed for every initial condition $\rho(0) = \rho^{\circ}\in\R$. However, for monotone distributed routing policies, existence of a limit density is ensured for the initial condition $\rho(0)=\zerobf$. 
Indeed, for every $\rho^{\circ}\in\mc R$ and $0\le t<\kappa(\rho^{\circ})$, let $\phi^t(\rho^{\circ})=\rho(t)$ be the solution of \eqref{dynsysexample} with initial condition $\rho(0)=\rho^{\circ}$. Then, for monotone distributed routing policies, \eqref{eq:monotonicity} implies that $\phi^{t+s}(\zerobf) = \phi^t(\phi^s(\zerobf)) \geq \phi^t(\zerobf)\,,$ for $0\le t<\kappa(\rho^{\circ})$ and $0\le s<\kappa(\rho^{\circ})-t$, 
 i.e., $\phi^t(\zerobf)$ is component-wise non-decreasing and hence convergent to some limit, to be denoted, with slight abuse of notation, by $\rho^*:=\lim_{t\to\kappa(\zerobf)}\phi^t(\zerobf)\,.$ 

 Let $\mc B$, $\mc W$, $\mc C$, $\Zi$, and $\Zo$ be defined as in \eqref{BWCF} for $\rho^{\circ}=\zerobf$.
First, consider the case $\max_{\U}\left(\lambda_{\mc U}-C_{\mc U}\right)<0$. Then, Lemma \ref{lemma:existenceOfCut} implies that $\E = \W$, as otherwise there would exist a cut $\S$ such that $C_{\S} \leq \lambda_{\S}$. Then, $\rho^*$ is an equilibrium. For an arbitrary initial condition $\rho^{\circ}\in\R$, it cannot be that $\kappa(\rho^{\circ})<\infty$, as then the limit $\lim_{t\uparrow\kappa(\rho^{\circ})}\phi^t(\rho^{\circ})\notin\mc R$ would exist, and Lemma \ref{lemma:existenceOfCut} would imply that $\lambda_{\mc S}\ge C_{\S}$ for some cut $\S$. Therefore, $\kappa(\rho^{\circ})=\infty$, for all $\rho^{\circ}\in\R$. By Lemma \ref{corollary:contraction}, we also have $||\phi^t(\rho^{\circ})-\rho^*||_1\le||\rho^{\circ}-\rho^*||_1$, for all $t\ge0$, so that in particular $\phi^t(\rho^{\circ})$ remains bounded. If the distributed routing policy is strongly monotone, then Lemma \ref{corollary:contraction} allows one to use LaSalle's theorem showing that $\lim_{t\to\infty}\phi^t(\rho^{\circ})=\rho^*$ for any initial condition $\rho^\circ\in\mc R$.

Conversely, if $\rho^*\in \R$, then, for every cut $\U$, mass balance on $\E^+_{\U}$ implies that 
	$$
		0 = \lambda_{\U}-\sum\nolimits_{e\in\partial_{\U}^+}\fout_e(\rho^*) + \sum\nolimits_{e\in\partial_{\U}^-}\fout_e(\rho^*) \ge\lambda_{\U}-C_{\U}\,.
	$$
This proves that, if $\lambda_{\U}>C_{\U}$ for some cut $\U$, then necessarily $\rho^*\notin \R$.   
The same holds if $\max_{\U}\left\{\lambda_{\U}-C_{\U}\right\}=0$ and the routing policy is strongly monotone, for in that case $\sum_{e\in\partial_{\U}^+}\fout_e(\rho^*)<C_{\U}$ if $\rho^*\in\R$. Therefore, $\mc W \neq \mc E$, so that Lemma \ref{lemma:existenceOfCut} implies \eqref{maintheo:claim2} for $\rho^{\circ}=\zerobf$. For arbitrary initial density $\rho^{\circ}\in\R$, consider the following two cases: $\kappa(\rho^{\circ})<+\infty$ and $\kappa(\rho^{\circ})=+\infty$. In the former, $\lim_{t\uparrow\kappa(\rho^{\circ})}\rho(t)$ exists, hence \eqref{maintheo:claim2} is implied by Lemma~\ref{lemma:existenceOfCut}. In the latter, $\kappa(\zerobf) \geq \kappa(\rho^{\circ})=\infty$,  hence \eqref{maintheo:claim2} for $\rho^{\circ}=\zerobf$ also implies \eqref{maintheo:claim2} for arbitrary $\rho^{\circ} \in \R$.
 
\subsection{Proof of Proposition \ref{corollary:FiniteBuffer}}

Observe that, for every cut $\U$, 
	$$
	\sum\nolimits_{e\in\E_\U^+}\dot{\rho}_e 
		= 
			\lambda_\U + \sum\nolimits_{e\in\partial_{\U}^-}\fout_e - 
			\!\sum\nolimits_{e\in\partial_{\U}^+}\!\fout_e \geq \lambda_\U - C_\U\,,  
	$$
so that $\sum_{e\in\E_\U^+}{\rho}_e\ge\sum_{e\in\E_\U^+}{\rho}^\circ_e+t( \lambda_\U - C_\U)$, from which \eqref{kappaUB} follows. On the other hand, \eqref{Sexplodes} is an immediate consequence of claim ii) of Theorem \ref{theorem:mainResult} and the definition of $\kappa(\rho^\circ)$.\\ \vspace{-.6cm}

\subsection{Proof of Proposition \ref{proposition:InfiniteBuffer}}

\newmaterial{
Let $\U^*$ be defined as in \eqref{U*def}, and $\S$ be a cut whose existence is guaranteed by Lemma \ref{lemma:existenceOfCut} for $\rho^\circ=\zerobf$. 
The proof consists of three steps: 1) Lemma~\ref{lem:maximal-cut} characterizes $\U^*$ defined in \eqref{U*def} as the maximal cut such that $\lambda_{\U^*} - C_{\U^*} = \max_{\U}\left\{\lambda_\U - C_{\U}\right\} \geq 0$. 2) Lemma~\ref{lemma:S=U} shows that $\S = \U^*$, where $\S$ is the cut built in Lemma \ref{lemma:existenceOfCut} for $\rho^\circ=\zerobf$. 3) The proof is completed for $\rho^\circ=\zerobf$ and extended to the case of generic initial condition.

\begin{lemma}
\label{lem:maximal-cut}
For a network $\mc G=(\mc V,\mc E,C)$ satisfying Assumption \ref{assumption:connectivity}, let $\mc U^*$ and $\mc M$ be as in \eqref{U*def}. Then, $\mc U^*\in\mc M$.
\end{lemma}
\begin{proof}
We will prove that $\U_1\cup\U_2\in \mc M$ for $\mc U_1,\mc U_2\in\mc M$. For $\A, {\mc H} \subseteq \V$, let $C^{\A}_{\mc H} := \sum_{e:\sigma_e \in \A, \tau_e \in {\mc H}}C_e$. It is easy to see that
\begin{equation}
	\label{equation:instrumentalEquality}
	\lambda_{\A\cup {\mc H}} - C_{\A\cup {\mc H}}
	=\lambda_{\A} + \lambda_{{\mc H}\setminus \A} - C_{\A} + C^{\A}_{\mc H\setminus \A} - C^{{\mc H}\setminus \A}_{\mc V\setminus(\A\cup \mc H)}\,.
\end{equation}
For  $\U_1,\U_2\in\mc M$,  put $\mc I:=\U_1\cap\U_2$, $\mc J:=\U_1\cup\mc U_2$, $\K:=\mc U_2\setminus\mc U_1$. Observe that $\lambda_{\mc J} - C_{\mc J} \leq \lambda_{\U_1} - C_{\U_1}$ since $\U_1\in\mc M$. We now prove that $\lambda_{\mc J} - C_{\mc J} \geq \lambda_{\U_1} - C_{\U_1}$. Assume by contradiction that	$\lambda_{\mc J} - C_{\mc J} < \lambda_{\U_1} - C_{\U_1}= \lambda_{\U_2} - C_{\U_2}.$ Then, \eqref{equation:instrumentalEquality} with $\A = \U_1$ and $\mc H = \U_2$ gives 
$$\lambda_{\U_1} + \lambda_{\mc K} - {\, C_{\U_1}} + C^{\U_1}_{ \mc K} - C^{\mc K}_{\V\setminus\mc J}<\lambda_{\U_1} - C_{\U_1}$$
which yields
	\begin{equation}
		\label{equation:instrumentalContradiction}
		\lambda_{\mc K} + C^{\U_1}_{\mc K} - C^{\mc K}_{ \V\setminus\mc J} < 0\,.
	\end{equation}
Similarly, applying \eqref{equation:instrumentalEquality} with $\A = \mc K$ and $\mc H = \mc I$, noting that $\mc K\cap\mc I=\emptyset$, and using $C_{\mc K} 		= 		C^{\mc K}_{\V\setminus\mc J}		+C^{\mc K}_{\U_1}$ yields
	\begin{equation}
	\label{eq:lambda2-C2}
		\lambda_{\U_2} - C_{\U_2}
		= 
	 \lambda_{\mc K} + \lambda_{\mc I} - C^{\mc K}_{ \V\setminus\mc J}  - C^{\mc K}_{\U_1} + C^{\mc K}_{ \mc I} - C^{\mc I }_{ \V\setminus\U_2}.	
	\end{equation}
Combining \eqref{eq:lambda2-C2} and \eqref{equation:instrumentalContradiction}, some algebraic steps lead to 	
	$$
	\ba{rclcl}\lambda_{\U_2} - C_{\U_2}\!\!
			 &\!\!\!<\!\!\!&	 \lambda_{\mc I} - C^{\U_1}_{ \mc K} - C^{\mc K}_{\U_1} + C^{\mc K}_{ \mc I} - C^{\mc I }_{ \V\setminus\U_2}&&\\	
& \!\!\!=\!\!\!& \lambda_{\mc I} - C_{\mc I}-C^{\U_1\setminus\U_2}_{ \mc K}
		   - C^{\mc K}_{\U_1\setminus\U_2} &\!\!\!\!\!\!<\!\!\!&\!\!\!	 
 \lambda_{\mc I} - C_{\mc I}.\ea$$
Hence, $\lambda_{\mc I} - C_{\mc I} > \lambda_{\U_2} - C_{\U_2}$, which contradicts ${\U_2}\in\mc M$. This proves that $\lambda_{\mc J} - C_{\mc J} =  \lambda_{\U_2} - C_{\U_2} = \lambda_{\U_1} - C_{\U_1}$. 
\end{proof}

\begin{lemma}
\label{lemma:S=U}
Let $\mc G=(\mc V,\mc E,C)$ be a network satisfying Assumption \ref{assumption:connectivity} and $\lambda$ a vector of inflows such that $\max_{\U}\left\{\lambda_{\U}-C_{\U}\right\} \ge0$. Let $f$ be a strongly monotone distributed routing policy. Let $\U^*$ be defined as in \eqref{U*def} and $\mc B,\mc W,\mc C,\Zo\subseteq\mc E$ be defined as in \eqref{BWCF} for $\rho^{\circ}=\zerobf$. If $\kappa(\zerobf)=+\infty$, then 
$\mc E_{\mc U^*}^+ \subseteq \B$, $\partial_{\U^*}^+ \subseteq \C$, $\partial_{\U^*}^- \subseteq \Zo$, and $\E\setminus(\E^+_{\U^*}\cup\partial_{\U^*}^-)\subseteq\W$.
\end{lemma}
\begin{proof}
Let $\rho(t)$ be the solution of \eqref{dynsysexample} with initial condition $\rho(0)=\zerobf$ and $\S:=\{v\in\mc V\setminus \D:\,\E^+_v\subseteq\B\}$. Observe that, as argued in Sect.~\ref{sec:proofmaintheo}, $\dot\rho_e=\fin_e(\rho)-\fout_e(\rho^e)\ge0$ for all $e$, so that in particular $\Z_i\subseteq\Z_o$. On the other hand, Barbalat's lemma implies that $\dot\rho_e\to0$ for $e\in\mc W$, so that $\mc W\cap\Zo\subseteq\Z_i$. Then, it follows from Lemma \ref{lemma:existenceOfCut} that $\partial_\S^+\subseteq\C$, $\partial_\S^-\subseteq\Zo\cap\Zi$, and $\E\setminus(\E^+_\S\cup\partial_\S^-)\subseteq\W$.

It remains to show that $\S = \U^*$. We start by proving that $\S\subseteq\U^*$. Define $\mc H:=\S\setminus\U^*$, $\mc I:=\partial_{\mc H}^-\cap\E^+_{\S}$, and $\mc J:=\partial_{\mc H}^+\cap\partial_{\S}^+$. Then, 
	\begin{align*}
		0	&\le \sum_{e\in\E^+_{\mc H}}\dot{\rho}_e(t) \\
			&\le  \lambda_\mc H + \sum_{e\in\partial_\S^-}\fout_e(t) 	
			+\sum_{i\in\mc I}\fout_i(t)
  - \sum_{j\in\mc J}\fout_j(t)\,.
 	\end{align*}
Passing to the limit of large $t$, $\partial_\S^- \subseteq\Zo$ and $\partial_\S^+\subseteq\C$ imply 
$$
		0\le\lambda_\mc H + \sum_{i\in\mc I}\fmax_i -\sum_{j\in\mc J}\fmax_j\,.
$$

Let now $\hat{\U} := \S\cup\U^* \supseteq \U^*$ and notice that $\mc K:=\partial_{\mc H}^+ \setminus\partial_{\U^*}^-\subseteq \mc J$ and $\mc I\subseteq\partial_\mc H^- \cap\partial_{\U^*}^+=:\mc L$. Then,
$$
	C_{\hat{\U}}=C_{{\U}^*}+c_{\mc K}-c_{\mc L}
	  \le C_{{\U}^*}+c_{\mc J}-c_{\mc I}
 \le C_{{\U}^*}+\lambda_\mc H\,,
$$
where $c_{\mc X}:=\sum_{x\in\mc X}\fmax_c$ for $\mc X=\mc I,\mc J,\mc K,\mc L$.
This implies that
	$$
		\lambda_{\hat \U}-C_{\hat \U}=\lambda_{\U^*}+\lambda_{\mc H}-C_{\hat \U}\ge \lambda_{\U^*}-C_{{\U}^*}\,,
	$$
so that $\hat \U\in\mc M$, and then $\hat \U=\U^*$. Therefore, $\S\subseteq \U^*$. 
We now prove that $\U^*\subseteq\S$. 
Assume by contradiction that $\A:=\U^*\setminus\S\neq\emptyset$.	Let  
	$$
			\Upsilon:=\lambda_{\A} +\!\!\sum_{e\in\partial_\A^- \cap\partial_\S^+}\!\!\fmax_e 
			  \, + \liminf_{t}\!\!\sum_{k\in\partial_\A^-\setminus\partial_{\mc S}^+}\!\!\fout_k(t)-\!\!\sum_{j\in\partial_\A^+ \setminus\partial_\S^-}\!\!\fout_j(t).
	$$
Then, the inclusions $\partial_\S^- \subseteq\Zo\cap\Z_i$ and $\partial_{\S}^+\subseteq\C$ imply
	$$
		\ba{lcl}
			\ds\liminf _t\!\!\! \sum_{e\in\E_{\A}^+\setminus\partial_\S^-}\dot\rho_e(t)\!\!\!\!\!
				& =
				\ds\liminf_t\!\!\!\sum_{e\in\E_{\A}^+\setminus\partial_\S^-}\!\!\!\! \left( \fin_e(t)-\fout_e(t) \right)&\\[15pt]\ds
				 &=
				\ds\liminf_t\sum_{e\in\E_{\A}^+} \left( \fin_e(t)-\fout_e(t) \right) &
				 \!\!\!= \Upsilon.
		\ea
	$$	
Observe that strict monotonicity implies  that 
	\be
		\label{strictmono}
			\limsup_t \fout_j(t)<C_j\,,\qquad\liminf_t\fout_{k}(t) > 0\,,
	\ee
for all $j\in\mc\E^+_\A\setminus\partial_\S^-$ and $k\in\partial_\A^-$. If $\partial_\A^-\setminus\partial_{\mc S}^+ = \partial_\A^+ \setminus\partial_\S^- = \emptyset$, then Assumption \ref{assumption:connectivity} implies that $\lambda_\A > 0$ or $\partial_\A^- \cap\partial_\S^+ \neq \emptyset$, therefore $\Upsilon =\lambda_{\A} +\!\!\sum_{e\in\partial_\A^- \cap\partial_\S^+}\!\!\fmax_e >0$. 

On the other hand if $\partial_\A^-\setminus\partial_{\mc S}^+ \neq \emptyset$ or $\partial_\A^+ \setminus\partial_\S^- \neq \emptyset$, then \eqref{strictmono} and $\S\subseteq\U^*$ imply
\begin{align*}
	&\Upsilon
			>\lambda_{\A}+\!\!\!\!\!\sum_{e\in\partial_\A^- \cap\partial_\S^+}\!\!\!\!\!\fmax_e - \!\!\!\!\!\sum_{e\in\partial_\A^+ \setminus\partial_\S^-}\!\!\!\!\!C_e
			=
			\lambda_{\U^*}-\lambda_{\S}-C_{\U^*}+C_{\S} \ge0\,,
\end{align*}	
the last inequality holding since $\U^*\in\mc M$ by Lemma \ref{lem:maximal-cut}. In both cases, $\liminf_t \sum_{e\in\E_{\A}^+\setminus\partial_\S^-}\dot\rho_e(t) = \Upsilon > 0$, which contradicts $\E_{\A}^+\setminus\partial_\S^- \subseteq\W$. 
Then, necessarily $\A=\emptyset$, so that $\U^*\subseteq\S$.
\end{proof}\medskip

We can now conclude the proof of Proposition \ref{proposition:InfiniteBuffer}. Infinite buffers and limited growth rate imply $\kappa(\rho^\circ) = \infty$ for every $\rho^\circ \in\R$. For $\rho^\circ=\zerobf$, Lemmas \ref{lemma:existenceOfCut} and \ref{lemma:S=U} imply \eqref{limE*} and \eqref{limE*S}. For arbitrary $\rho^{\circ}\in\R$, the extension of \eqref{limE*} follows from Lemma \ref{corollary:contraction}, hence we only need to prove \eqref{limE*S}. 
Towards this goal, first note that $\rho^\circ \geq \zerobf$ implies, by monotonicity, that
	\be
		\label{mono1}\liminf_{t\to\infty}\rho_e(t)\ge\rho_e^*\,,\qquad \forall e\in\hat\E\,.
		\ee 
Consider a new network $\hat\G=(\hat\V,\hat\E,\hat C)$ with $\hat\V:=\mc V\setminus\S$ and $\hat C_e=C_e$ for $e\in\hat\E$, and with inflows $\hat\lambda_{\hat v}:= \lambda_{\hat{v}} + \sum_{e\in\mc E^-_{\hat v}\cap\partial_\S^+}C_e$ for $\hat v\in\hat\V$, and buffer capacities $\hat B_e=B_e$ for $e\in\hat\E$. Let $\hat f$ be a distributed routing function for $\hat\G$ such that $\hat f_{e\to j}(\hat\rho^e)=f_{e\to j}(\rho^e)$ where $\rho^e\in\R^\bullet_e$ is such that $\rho_j=\hat\rho_j$ for all $j\in\E_e\cap\hat\E$, and $\rho_j=B_j$ for all $j\in\E_e\cap\partial_{\S}^-$. This defines a dynamics on the reduced network $\hat\G$.

Observe now that clearly $\hat\G$ satisfies Assumption~\ref{assumption:connectivity}. In addition, $\S=\U^*$ implies $\hat{\lambda}_{\hat\U} < C_{\hat\U}$ for every cut $\hat \U$ in $\hat \G$, where $\hat\lambda_{\hat\U} = \sum_{\hat v\in\hat\U}\hat\lambda_{\hat v}$. Then, applying part i) of Theorem~\ref{theorem:mainResult} to the dynamical flow network associated to $\hat \G$ and $\{\hat f_e\}_{e\in\hat\E}$ shows existence of a globally attractive equilibrium, $\hat{\rho}^*=\lim_t\hat\rho(t)$. Notice that the solution to this system coincides with the solution of the original one once we fix to be equal to $B_e = +\infty$ the density $\rho_e$ for every $e\in\E^+_\S\cup\partial_{\U^*}^-$. In particular, asymptotically, the limits must be the same, i.e., $\hat\rho^*=\rho^*_{\hat\E}$.

Finally, the new network is a monotone controlled system \cite{AngeliTAC03}, once we interpret the densities on $\E^+_\S\cup\partial_\S^-$ as inputs. Since $\rho_e(t) < B_e = +\infty$ for all $e \in \E^+_\S\cup\partial_\S^-$ and $t \geq 0$, one gets that
	\begin{equation}
	\label{equation:notBlowingAbove}
		\limsup_{t\to\infty}\rho_e(t) \leq \lim_{t\to\infty}\hat \rho_e(t) = \rho_e^*\,,\qquad \forall e\in\hat\E\,.
	\end{equation}
Combining \eqref{mono1} and \eqref{equation:notBlowingAbove} gives \eqref{limE*S} for arbitrary $\rho^\circ\in\R$.


}

\section{Conclusion}
\label{section:conclusions}
We study dynamical flow networks under distributed monotone routing policies. An $l_1$-contraction argument for monotone systems is instrumental to prove throughput optimality of the proposed policies both when the min-cut capacity constraints are satisfied and in overload. These tools can be fruitfully employed for analysis of transportation networks \cite{LovisariCDC14}.



Future research includes and is not limited to design of application-oriented control policies and optimization with respect to secondary objectives, such as steady-state delay, without compromising throughput optimality. We also plan to extend our framework to the multi-commodity case under partial state feedback, modeling urban traffic networks where observations are the aggregates of flows of all commodities.




\bibliographystyle{IEEEtran}
\bibliography{bibliography}             

\end{document}